\documentclass[10pt,reqno]{article}

\usepackage{setspace}
\usepackage{amsmath,amssymb,amsthm}
\usepackage{esint}
\usepackage{bm}
\usepackage{url}
\usepackage{bbm}
\usepackage{subfigure}
\usepackage{graphicx,color}
\usepackage{algorithm}
\usepackage{algpseudocode}
\usepackage{algorithmicx}

\usepackage{booktabs,multirow}
\usepackage{hhline}
\usepackage{setspace}
\usepackage{caption}
\usepackage{hyphsubst}
\usepackage{caption}

\usepackage{float}
\usepackage{filecontents}
\usepackage{hyperref}

\hypersetup{colorlinks,citecolor=blue,linkcolor=blue, urlcolor=blue}
\usepackage{pdflscape}
\usepackage{breakurl}

\DeclareMathAlphabet{\itbf}{OML}{cmm}{b}{it}
\def\by{{{\bf y}}}
\def\bx{{{\bf x}}}
\def\bz{{{\bf z}}}
\def\bpsi{{\boldsymbol{\psi}}}
\def\bpi{{\boldsymbol{\pi}}}

\def\bphi{{\boldsymbol{\varphi}}}

\def\be{{\hat{\itbf e}}}
\def\bu{{{\bf u}}}
\def\bv{{{\bf v}}}
\def\bw{{{\bf w}}}
\def\Bx{{\bf x}}
\def\By{{\bf y}}

\def\bs{{{\bf s}}}
\def\Ba{{\bf a}}
\def\Bb{{\bf b}}

\newcommand{\ds}{\displaystyle}
\newcommand{\RR}{\mathbb{R}}
\newcommand{\C}{\mathcal{C}}
\newcommand{\CC}{\mathbb{C}}
\newcommand{\nm}{\smallskip}
\newcommand{\bp}{\mathbf{p}}
\newcommand{\bg}{\mathbf{g}}
\newcommand{\bnu}{\bm{\nu}}
\newcommand{\OL}{\mathcal{L}}
\newcommand{\bU}{\mathbf{U}}
\newcommand{\bGam}{\mathbf{\Gamma}}
\newcommand{\I}{\mathbf{I}}
\newcommand{\Kcal}{\mathcal{K}}
\newcommand{\NN}{\mathbf{N}}
\newcommand{\bLambda}{\mathbf{\Lambda}}
\newcommand{\OR}{\mathcal{R}}
\newcommand{\bG}{\mathbf{G}}
\newcommand{\br}{\mathbf{r}}
\newcommand{\bfeta}{\boldsymbol{\eta}}
\newcommand{\D}{\mathcal{D}}
\newcommand{\f}{\mathbf{f}}

\newtheorem{thm}{Theorem}[section]

\newtheorem{lem}[thm]{Lemma}

\newcommand{\email}[1]{\protect\href{mailto:#1}{#1}}

\numberwithin{equation}{section}

\newcommand{\pathfigures}{Figures/}
\graphicspath{{\pathfigures}}
\usepackage[english]{babel}

\begin{document}
\title{A Joint Sparse Recovery Framework for Accurate Reconstruction of Inclusions in Elastic Media\thanks{This research was supported by the Ministry of Science, ICT and Future Planning through the National Research Foundation of Korea grants  NRF-2016R1A2B3008104 (to J.Y., J.C.Y., and A.W.), NRF-2014R1A2A1A11052491 (to J.Y., J.C.Y., and A.W.),  NRF-2016R1A2B4014530 (to Y.J., and M.L.), NRF-2015H1D3A1062400 (to A.W. through the Korea Research Fellowship Program) and R\&D Convergence Program of National Research Council of Science and Technology of Korea grant no. CAP-13-3-KERI (to J.Y., J.C.Y., and A.W.).}
}
\author{
Jaejun Yoo\footnotemark[2]\ \footnotemark[4]
\and
Younghoon Jung\footnotemark[3]\ \footnotemark[4]
\and 
Mikyoung Lim\footnotemark[3]\ \footnotemark[5]
\and 
Jong Chul Ye\footnotemark[2]\ \footnotemark[5]
\and 
Abdul Wahab\footnotemark[2]\ \footnotemark[6]
}
\maketitle
\renewcommand{\thefootnote}{\fnsymbol{footnote}}
\footnotetext[2]{Bio Imaging and Signal Processing Lab., Department of Bio and Brain Engineering, Korea Advanced Institute of Science and Technology, 291 Daehak-ro, Yuseong-gu, Daejeon 34141, Republic of Korea (\email{jaejun2004@kaist.ac.kr}; \email{jong.ye@kaist.ac.kr}; \email{wahab@kaist.ac.kr}).}
\footnotetext[3]{Department of Mathematics, Korea Advanced Institute of Science and Technology, 291 Daehak-ro, Yuseong-gu, Daejeon 34141, Republic of Korea (\email{hapy1010@kaist.ac.kr}; \email{mklim@kaist.ac.kr}).}
\footnotetext[4]{J. Yoo and Y. Jung contributed equally to this work and are co-first authors.}
\footnotetext[5]{M. Lim and J. C. Ye are co-corresponding authors.}
\footnotetext[6]{Address all correspondence to A. Wahab  at \email{wahab@kaist.ac.kr}, Ph.:+82-42-3504320, Fax:+82-42-3504310.}
\renewcommand{\thefootnote}{\arabic{footnote}}
\begin{abstract}
A robust algorithm is proposed to reconstruct the spatial support and the  Lam\'e parameters of multiple inclusions in a homogeneous background elastic material using a few measurements of the displacement field over a finite collection of boundary points.   The algorithm does not require any linearization or iterative update of Green's function but still allows very accurate reconstruction. The breakthrough comes from a novel interpretation of Lippmann-Schwinger type integral representation of the displacement field in terms of unknown densities having common sparse support on the location of inclusions.
Accordingly, the proposed algorithm consists of a two-step approach. First, the localization problem is recast as a joint sparse recovery problem that renders the densities and the inclusion support simultaneously. Then, a noise robust constrained optimization problem is formulated for the reconstruction of elastic parameters. 
An efficient algorithm is designed for numerical implementation using the  Multiple Sparse Bayesian Learning (M-SBL) for joint sparse recovery problem and the Constrained Split Augmented Lagrangian Shrinkage Algorithm (C-SALSA) for the constrained optimization problem. The efficacy of the proposed framework is manifested through extensive numerical simulations.  To the best of our knowledge, this is the first algorithm tailored for parameter reconstruction problems in elastic media using highly under-sampled data in the sense of Nyquist rate.
\end{abstract}



\section{Introduction}

Elasticity imaging or elastography is a set of thriving non-invasive imaging techniques that have led to significant improvements  in the quantitative evaluation and visualization of mechanical properties of elastic materials \cite{Princeton,  Greenleaf}. It aims to recover spatial variations in certain material and geometric parameters of structures inside an elastic body from displacement data obtained non-invasively over a part of the boundary surface or inside the body using classical imaging modalities such as ultrasound, magnetic resonance, or speckle interferometry \cite{Oberai, Sinkus00, Sinkus05, Sinkus06}. Different terminologies (static, quasi-static, time-harmonic, and dynamic elasticity imaging) are used to differentiate techniques based on excitation mechanism adapted to probe the underlying elastic body \cite{Parker}.

Elasticity imaging frameworks cater to a broad range of applications, for example,  non-destructive testing of elastic objects for material impurities and structural integrity \cite{Kang16}, exploration geophysics for mineral reservoir prospecting \cite{Tarantola, Weglein}, and medical diagnosis, in particular, for detection and characterization of potential tumors of diminishing sizes \cite{Sinkus00, Sinkus05,  Oberai}. In the perspectives of medical diagnosis  elasticity imaging aims to fathom spatial variations in the material parameters of human tissues by harnessing the interdependence between elastic field  and tissue elasticity. It can be perceived as a modernization of tissue palpation technique that has been used for centuries to identify abnormalities \cite{Sinkus06}.   In fact, the correlation between changes in the stiffness of tissues  with pathological phenomena, such as cirrhosis of the liver \cite{Cirrhosis}, weakening of vessel walls, and recruitment of collagen during tumorigenesis \cite{Garra, Hiltawsky, Wellmann}, has given an impetus to the quantitative  characterization of underlying elastic properties using modern apparatus.  

The inverse problem of quantitative evaluation of constitutive parameters is notorious  for its complexity and ill-posed character. Many dedicated mathematical and computational  algorithms for the reconstruction of location and parameters of anomalies of different geometrical nature (cavities, cracks, and inclusions) have been proposed over the past few decades (see, for instance, \cite{Abbas, TDelastic, TrElastic, Bal14, Bal15a, Bal15b, Barbone, Constantinescu,Kang16,  Eskin, Geymonat, Guchhait, Ikehata, Imanuvilov,  Nakamura94, Nakamura03, Widlak},  the survey articles \cite{Bonnet, Bonnet08}, and the monograph \cite{Princeton}).  Most of the classical techniques are suited to continuous measurements, in other words, to experimental setups allowing to measure continuum deformations inside the elastic body or on a substantial part of its boundary. In practice, this requires mechanical systems that furnish discrete data sampled on a very fine grid confirming to the Nyquist sampling rate. Unfortunately, this is not practically feasible due to mechanical, computational and financial constraints. On the other hand,  several algorithms are based on linearizations with respect to the leading order of the scale factor of inclusions (for instance, asymptotic expansion methods \cite{TDelastic, Garapon}),  or the variations in the constitutive parameters. Born, Rytov and Foldy-Lax type approximations are also adopted. These simplifications are not always valid and are too strong to allow an accurate reconstruction. This results in a dramatic loss of image resolution and quality. The algorithms avoiding such assumptions usually require iterative updates and only a handful of direct reconstruction algorithms can be found in the literature.  Specifically, these techniques are computationally very costly and are highly prone to instabilities as they require computation of numerous forward solutions for iterative updates and suffer from intrinsic ill-posedness of the problem \cite{Monard}. In a nutshell, the existing results found in the literature are clearly not satisfactory from a practical point of view.

In this work,  an accurate novel imaging algorithm is proposed for the reconstruction of multiple inclusions present in a bounded isotropic homogeneous elastic formation. It is assumed that a few measurements of the displacement field over a small finite set of boundary points are available.  For simplicity,  an elastostatic regime is considered, however,  the quasi-static and time-harmonic elasticity problems are amenable to the same treatment with minor changes.
One of the most important features of the proposed algorithm is that it does not require any linearization or iterative update of the Green's function, yet it is felicitous to furnish the spatial support of the inclusions and their material parameters very accurately. The breakthrough comes from a novel interpretation of the Lippmann-Schwinger type integral representation of the displacement field that is derived in terms of unknown densities having jointly sparse spatial support on the location of inclusions.  Therefore, the support identification problem can be recast as a joint sparse recovery problem for the unknown densities given that the support set of inclusions is itself \emph{sparse} inside the elastic formation. This allows invoking a variety of compressed sensing signal recovery algorithms. Consequently, using any one of these algorithms, the solution of the joint sparse recovery problem can be obtained which yields not only the spatial support of the inclusions but also renders the unknown densities. 
The Lam\'e parameters of the inclusions are estimated using recovered densities in the second step of the proposed imaging framework by solving a linear inverse problem in sought parameters.
 
It is worthwhile mentioning that the additional information contained in the recovered densities is linked to the perturbed displacement and strain fields inside the support of the inclusions. The availability of the internal data and the sparsity assumption on the support of the inclusions thus compensate for the lack of over-determined measurements and significantly reduce the ill-posedness of the problem.  In particular, this paves the way to a resolution enhancement since no linearization is applied in the proposed algorithm. Moreover, the numerical implementation of the proposed technique, as will be discussed later on in Section \ref{s:Implement},  does not require multiple forward solutions and is computationally very efficient.  
In fact, similar two-step approaches using joint sparse recovery formulations have been previously developed by our group for inverse scattering problems related to scalar Helmholtz equation \cite{Helmholtz},  diffuse optical tomography \cite{Lee2011CDOT, Lee2013DOT}, and electric impedance tomography \cite{EIT}. An extension to the electromagnetic inverse wave scattering governed by full Maxwell equations is on-going and has also provided very promising preliminary results. This extension will be reported elsewhere. 

An important aspect of this work is the manifestation that even for elastic scattering problems, despite their complexity and ill-posedness,  the derivation of an integral representation in terms of jointly sparse densities is possible that leads to an accurate and stable reconstruction with highly under-sampled data, which is not generally tractable by classical techniques. 
This clearly indicates, in view of the present investigation and the previously obtained results \cite{Helmholtz, Lee2011CDOT, Lee2013DOT}, that the proposed formulation may be so general that it provides a unified reconstruction framework for assorted inverse scattering problems.   

The contents of this article are organized in the following order.  The mathematical formulation of the inverse problem is provided in Section \ref{s:Form}.  The Lippmann-Schwinger type integral representation of the displacement field is derived in Section \ref{s:int}. Section \ref{s:Recon} is dedicated to the joint sparse recovery based reconstruction algorithm. The computational aspects of the algorithm are discussed in Section \ref{s:Implement}. Several numerical experiments are conducted to substantiate the appositeness of the proposed technique in Section \ref{s:NS}. The article ends with a  summary of this investigation and a brief discussion provided in Section \ref{s:conc}.

\section{Mathematical formulation}\label{s:Form}

In this section, the nomenclature and assumptions are specified that are adopted throughout this article and the mathematical formulation of the inverse problem dealt with in this article is provided.

\subsection{Preliminaries and nomenclature}\label{ss:Pre}

Consider an open bounded domain $\Omega\subset\RR^d$, $d=2,3$, with connected $\C^2$-boundary $\partial\Omega$. Let us define $L^2(\Omega)$, $L^\infty(\Omega)$, $H^1(\Omega)$ and   $H^2(\Omega)$ in the usual way endowed with standard norms. Let
 $H^{3/2}(\Omega)$ be defined as the  interpolation space $[H^1(\Omega), H^2(\Omega)]_{1/2}$.  
%The interested readers are referred, for instance, to the monograph by Bergh and L\"ofstr\"om \cite{Bergh} for further details. 

To facilitate the latter analysis, the subspace $L_{\Psi}^2(\partial\Omega)$ of $L^2(\partial\Omega)^d$ is defined by 
$$
L^2_\Psi(\partial\Omega):=\left\{\bphi\in L^2(\partial\Omega)^d\,\Big|\,\int_{\partial\Omega}\bphi\cdot\bpsi d\sigma=0,\quad \forall \bpsi\in\Psi\right\},
$$
where $d\sigma$ denotes the infinitesimal surface element and  $\Psi$ is the $d(d+1)/2$ dimensional vector space of infinitesimal rigid displacements, i.e.,
$$
\Psi:=\left\{\bpsi\in H^1(\Omega)^d\,\Big|\, \partial_i\psi_j+\partial_j\psi_i=0,\quad 1\leq i,j \leq d\right\}.
$$
It is interesting to note that the vector space $\Psi$ contains constant functions and as a result
$$
\int_{\partial\Omega} \bphi d\sigma=0,\quad\forall \bphi\in L^2_\Psi(\partial\Omega).
$$

Let $\Omega$ be loaded with an isotropic homogeneous elastic material so that its stiffness tensor $\CC^0=(C^0_{ijkl})_{i,j,k,l=1}^d$ is defined by 
\begin{eqnarray*}
C^0_{ijkl}:=\lambda_0\delta_{ij}\delta_{kl}+\mu_0\left(\delta_{ik}\delta_{jl}+\delta_{il}\delta_{jk}\right),%\label{C_0}
\end{eqnarray*}
where $\delta_{ij}$ is the Kronecker's delta function, and  constants $\lambda_0$ and $\mu_0$ are respectively the compression and shear moduli of the elastic material. It is assumed that these Lam\'e parameters satisfy the conditions
\begin{eqnarray}
\mu_0>0\quad\text{and}\quad d\lambda_0+2\mu_0>0.\label{convexity}
\end{eqnarray}
In fact, the conditions in \eqref{convexity} ensure the strong convexity of the stiffness tensor $\CC^0$. Precisely, for all symmetric matrices $\mathbf{A}\in\RR^{d\times d}\setminus\{\mathbf{0}\}$,  
\begin{eqnarray*}
\left(\CC^0:\mathbf{A}\right):\mathbf{A}\geq \min(2\mu_0, d\lambda_0+2\mu_0) \|\mathbf{A}\|_F^2,
\end{eqnarray*}
where the double dot operator $``:"$ (with varying definitions for $2-$, $3-$, or $4-$rank tensors by abuse of notation) and the Frobenius norm $\|\cdot\|_F$ are defined by 
\begin{align*}
\mathbb{M}:\mathbf{A}:=\bigg(\sum_{k,l=1}^d m_{ijkl}a_{kl}\bigg)_{i,j=1}^d,
\quad
\mathbf{A}:\mathbf{B}:=\sum_{i,j=1}^d a_{ij}b_{ij},
\quad\text{and}\quad
\|\mathbf{A}\|_{F}:=\sqrt{\mathbf{A}:\mathbf{A}},
\end{align*}
for real matrices $\mathbf{A}=(a_{ij})_{i,j=1}^d$ and $\mathbf{B}=(b_{ij})_{i,j=1}^d$, and $4-$ rank tensors $\mathbb{M}=\left(m_{ijkl}\right)_{i,j,k,l=1}^d$.

Suppose that the material loaded in $\Omega$ contains $N\in\mathbb{N}$ open and bounded elastic inclusions $D_n$, $n=1,\cdots,N$, with simply connected  smooth boundaries $\partial D_n$. To simplify the matters, the following assumptions are made throughout in this investigation. 
\begin{enumerate}

\item[\textbf{H1.}] All inclusions are separated apart from $\partial \Omega$ and their closures are mutually disjoint, i.e., there exists a constant $d_0\in\RR_+$ such that 
\begin{equation*}%\label{D-Omega-dist}
\inf_{x\in \overline{D_n}}{\rm dist}(x,\partial\Omega)\geq d_0
\quad\text{and}\quad
\inf_{x\in \overline{D_m}}{\rm dist}(x,\overline{D_n})\geq d_0, \quad \forall m\neq n,
\end{equation*}
where \emph{dist} represents the usual distance function in $\RR^d$. 

\item[\textbf{H2.}] Each $D_n$ is isotropic but is allowed to be  inhomogeneous so that its stiffness tensor, $\CC^n(\bx):=(C^n_{ijkl}(\bx))_{i,j,k,l}^d$, is given by 
\begin{eqnarray*}
C^n_{ijkl}(\bx):= \lambda_n(\bx)\delta_{ij}\delta_{kl}+\mu_n(\bx)\left(\delta_{ik}\delta_{jl}+\delta_{il}\delta_{jk}\right),
\qquad
\forall\,\bx\in D_n,
%\label{C_n}
\end{eqnarray*}
in terms of spatially varying Lam\'e parameters $\lambda_n,\mu_n\in L^{\infty}(\Omega)$. 

\item[\textbf{H3.}] There exist constants $\underline{\lambda}_n, \overline{\lambda}_n, \underline{\mu}_n,\overline{\mu}_n\in\RR$ such that 
\begin{eqnarray*}
0<\underline{\mu}_n\leq \mu_n(\bx)\leq \overline{\mu}_n
\quad\text{and}\quad
0< d\underline{\lambda}_n+2\underline{\mu}_n\leq d\lambda_n(\bx)+2\mu_n(\bx)\leq d\overline{\lambda}_n+2\overline{\mu}_n,
\end{eqnarray*}
for all $\bx\in\overline{D_n}$ and $n\in\{1,\cdots,N\}$. By this assumption, It is ensured that, for all symmetric matrices $\mathbf{A}\in\RR^{d\times d}\setminus\{\mathbf{0}\}$ and $\bx\in D_n$,
\begin{eqnarray*}
\max\left(2\overline{\mu}_n, d\overline{\lambda}_n+2\overline{\mu}_n\right)
\|\mathbf{A}\|_{F}^2
\geq \; 
\left(\CC^n(\bx):\mathbf{A}\right):\mathbf{A}
\geq \; 
\min\left(2\underline{\mu}_n, d\underline{\lambda}_n+2\underline{\mu}_n\right) 
\|\mathbf{A}\|_{F}^2.
\end{eqnarray*}

\item[\textbf{H4.}] The degenerate cases are avoided hereinafter by assuming  that there exists a constant $\xi_0\in\RR_+$ so that
\begin{equation*}
(\lambda_0-\lambda_n(\bx))(\mu_0-\mu_n(\bx))\geq \xi_0,
\quad \forall\,\bx\in D_n, \, \,n=1,\cdots, N.
\end{equation*}
The assumption of uniform positivity is required to maintain the positivity of the elastic energy form involved in the Lippmann-Schwinger type representation of the scattered field presented in Section \ref{ss:IntEq}.
By this, it is also ensured that the stiffness tensor,
\begin{eqnarray*}
\CC^\star(\bx):=\CC^0\chi_{\Omega\setminus\cup_{n=1}^N \overline{D_n}}(\bx)+\sum_{n=1}^N \CC^n\chi_{D_n}(\bx), 
%\label{C*}
\end{eqnarray*}
of $\Omega$ in the presence of inclusions, $D_1,\cdots,D_N$, also satisfies strong convexity condition. Here $\chi_{D_n}$ represents the characteristic function of domain $D_n$. 
\qquad\endproof
\end{enumerate}

To facilitate latter analysis, let us introduce piece-wise defined functions:  
\begin{align*}
\lambda(\bx):=\sum_{n=1}^N \lambda_n(\bx)\chi_{D_n}(\bx),
\quad
\mu(\bx):=\sum_{n=1}^N \mu_n(\bx)\chi_{D_n}(\bx),
\quad\text{and}\quad
\CC(\bx):=\sum_{n=1}^N\CC^n(\bx)\chi_{D_n}(\bx),
\end{align*}
for all $\bx\in \ds\cup_{n=1}^N D_n$. Moreover, the following conventions will be used henceforth. Let $\mathbf{p}=(p_i)_{i=1}^d$, $\mathbf{A}=(a_{ij})_{i,j=1}^d$ and $\mathbf{B}=(b_{ijk})_{i,j,k=1}^d$ be respectively any arbitrary vector, a matrix and  a $3-$rank tensor. Let  $\mathbf{H}:\RR^d \to H^1(\Omega)^{d\times d} $, $\bx\mapsto (h_{ij}(\bx))_{i,j=1}^d$ be any matrix valued function  and $(\be_1,\cdots,\be_d)$ be the standard basis in $\RR^d$. Then, for all  $ i,j,k\in\{1,\cdots,d\}$,
\begin{align*} 
&
\mathbf{A}\cdot\mathbf{p}:=\sum_{i,j} a_{ij}p_j\be_i 
\quad\text{and}\quad
\mathbf{B}:\mathbf{A}:=\sum_{i,j,k} b_{ijk} a_{jk} \be_{i},
\\
&
\nabla\cdot\mathbf{H}(\bx):= \sum_{i,j}\left(\frac{\partial h_{ij}}{\partial x_i}(\bx)\right)\be_j,
\quad
\left[\nabla\mathbf{H}(\bx)\right]_{ijk}:=\frac{\partial h_{ij}}{\partial x_k}(\bx),
\quad\text{and}\quad
\left[\left(\nabla\mathbf{H}(\bx)\right)^\top\right]_{ijk}:=\frac{\partial h_{ik}}{\partial x_j}(\bx),
\end{align*}
where superposed $\top$ indicates the transpose operation and the notation $[\cdot]$ is used to denote any component of a tensor, matrix or vector (e.g., $[\bp]_2=p_2$ for a vector $\bp\in\RR^d$).
The operators ${slice}_1:\{1,\cdots,d\}\times \RR^{d\times d\times d}\to \RR^{d\times d}$ and $vec:\RR^{d\times d}\to \RR^{d^2}$, defined by 
\begin{align*}
&{\rm slice}_1(i,\mathbf{B}):=\left(b_{ijk}\right)_{j,k=1}^d \,\,\text{(for fixed $i$)} 
\quad \text{and}\quad
{\rm vec}(\mathbf{A}):= 
\begin{pmatrix}
\mathbf{a}_{1}^\top, & \cdots, & \mathbf{a}_{d}^\top 
\end{pmatrix}^\top,
\end{align*}
will be useful in the sequel. Here $\mathbf{a}_{i}$ denotes the $i-$th column of matrix $\mathbf{A}$. Finally, 
$$
{\rm diag}(r_1,\cdots,r_\ell)\in \RR^{\ell\times \ell}, \qquad \ell\in\mathbb{N},
$$
will represent a diagonal matrix with diagonal elements $r_1,\cdots, r_\ell\in\RR$.

\subsection{Problem formulation}\label{ss:Prob}

Let $\bu_m:\overline{\Omega}\to\RR^d$, for  each $m=1,\cdots, M\in\mathbb{N}$, be the displacement field in $\Omega$, in the presence of $D_1,\cdots, D_N$,  caused by an applied  surface traction $\bg_m\in L^2_\Psi(\partial\Omega)$ on its boundary $\partial \Omega$. Then, the vector field $\bu_m$ is the solution to 
\begin{equation}
\begin{cases}
\ds\nabla\cdot\left(\CC^\star:\mathcal{E}[\bu_m]\right)= \mathbf{0},  & \text{in}\;\;\Omega,
\\ 
\left(\CC^0:\mathcal{E}[\bu_m]\right)\cdot \bnu=\bg_m,  & \text{on}\;\; \partial \Omega,
\end{cases}
\label{Prob_Trans1}
\end{equation}
where $\bu_m\in H^1(\Omega)$ such that $\bu_m\big|_{\partial\Omega} \in L^2_\Psi(\partial\Omega)$ in order to ensure the existence of a unique weak solution. Here $\bnu$ is the outward unit normal to $\partial \Omega$ and $\mathcal{E}(\bu_m)$ denotes the strain tensor related to $\bu_m$ and is given by 
$
\mathcal{E}[\bu_m]:=(\nabla\bu_m+\nabla\bu_m^\top)/2.
$

For convenience,  the linear isotropic elasticity operator and the corresponding surface traction operator associated with $\CC^0$ will be usually denoted by $\OL_{\lambda_0,\mu_0}$ and ${\partial}/{\partial \bnu}$ respectively, i.e., for any smooth function  $\bw: \RR^d \to \RR^d$, 
\begin{align*}
&\OL_{\lambda_0,\mu_0}[\bw]:= \nabla\cdot\left(\CC^0:\mathcal{E}[\bw]\right),
\\
&\frac{\partial\bw}{\partial \bnu}:=  \left(\CC^0:\mathcal{E}[\bw]\right)\cdot\bnu=\lambda_0(\nabla\cdot\bw)\bnu+2\mu_0\mathcal{E}(\bw)\cdot\bnu.
%\label{Conormal_Def}
\end{align*}
Similarly, analogous notation, $\OL_{\lambda,\mu}$ and $\partial/\partial \widetilde{\bnu}$, will be used for operators corresponding to $\CC$ (i.e., to $\lambda$ and $\mu$). Moreover, a \emph{de facto} extension of these operators will be used  for matrix valued functions by invoking conventions listed in Section  \ref{ss:Pre}.

It can be easily verified that the problem \eqref{Prob_Trans1} is equivalent to the  transmission problem 
\begin{equation}
\label{Prob_Trans2}
\begin{cases}
\OL_{\lambda_0,\mu_0}[\bu_m]=0,  & \text{in }\Omega\setminus\cup_{n=1}^N\overline{D_n}, 
\\
\OL_{\lambda,\mu}[\bu_m]=0,  & \text{in } \cup_{n=1}^N{D_n}, 
\\ 
\bu_m{\big|_-}=\bu_m{\big|_+}, & \text{on }  \partial{D_n} \quad (\forall n=1,\cdots,N),
\\
\ds\frac{\partial\bu_m}{\partial\widetilde\bnu}{\Big|_-}=\ds\frac{\partial\bu_m}{\partial\bnu}{\Big|_+}, & \text{on } \partial{D_n}\quad (\forall n=1,\cdots,N),
\\
\ds\frac{\partial\bu_m}{\partial\bnu}=\bg_m, & \text{on }\partial\Omega 
\quad\Big(\bu_m\big|_{\partial\Omega}\in L^2_{\Psi}(\partial\Omega)\Big),
\end{cases}
\end{equation}
where subscripts $+$ and $-$ indicate the limiting values across the interface $\partial D_n$ from outside and from inside $D_n$ respectively, i.e., for any function  $\bw$,
$$
\bw\big|_\pm(\bx):=\lim_{\epsilon\to 0^+} \bw(\bx\pm \epsilon\bnu), \quad \bx\in\partial D_n.
$$

The background displacement field $\bU_m$ in $\Omega$ (in the absence of any inclusion), caused by the surface traction $\bg_m\in L^2_\Psi(\partial\Omega)$ applied on $\partial\Omega$, is also required. The vector field  $\bU_m:\overline{\Omega}\to\RR^d$ is the solution to 
\begin{equation*} 
\begin{cases}
\OL_{\lambda_0,\mu_0}[\bU_m]=0,  & \text{in}\;\; \Omega, 
\\ 
\ds\frac{\partial\bU_m}{\partial\bnu}=\bg_m, &\text{on}\;\;\partial\Omega 
\quad \Big( \bU_m\big|_{\partial\Omega}\in L^2_\Psi(\partial\Omega)\Big).
\end{cases}
\end{equation*}

The following inverse problem is dealt with in this article.  
\subsubsection*{Inverse Problem}
{\em
Let $\{\bx_{r}\}_{r=1}^R\subset\partial\Omega$, for some $R\in\mathbb{N}$, be a finite collection of points on $\partial\Omega$. Let a known traction $\bg_m\in L^2_\Psi(\partial\Omega)$, for each $m\in\{1,\cdots, M\in\mathbb{N}\}$, be applied on $\partial \Omega$ which induces the displacement fields $\bu_m$ and $\bU_m$ in $\Omega$, respectively, with and without the presence of inclusions, $D_1,\cdots,D_N$. Given the set of measurements
$$
\Big\{(\bu_m-\bU_m)(\bx_r)\quad\Big|\quad  r=1,\cdots, R,\; m=1,\cdots, M\Big\},
$$
locate $D_1,\cdots, D_N$ and reconstruct the corresponding Lam\'e parameters, $\lambda_n$ and $\mu_n$, for all $n=1,\cdots, N$. \qquad\endproof
}

 \section{Lippmann-Schwinger type integral representation of perturbed displacement} \label{s:int}

The main ingredient of our reconstruction framework is a Lippmann-Schwinger type integral representation of the perturbations in the displacement fields, $(\bu_m-\bU_m)$, due to the presence of inclusions, $D_1,\cdots, D_N$. In this section, an exact analytic formula is derived for the perturbations $(\bu_m-\bU_m)$ using tools mostly borrowed from the existing literature on integral equations. It is emphasized that similar Lippmann-Schwinger type formulations can be found in the literature; see, e.g., \cite{Snieder}. Nevertheless, the detailed derivation is provided in Section \ref{ss:IntEq} for completeness sake as this formulation is the key component of our algorithm. Towards this end, it is best to pause and recall a few elements from layer potential theory for the linear elastostatic system in Section \ref{ss:BIO}.  The readers interested in further details are invited to read, for instance, the recent monograph \cite{Princeton}.

\subsection{Elements of  layer potential theory for elastostatics}\label{ss:BIO}

Let $\bGam$ denote the Kelvin matrix  of fundamental solutions to the elastostatic system $\OL_{\lambda_0,\mu_0}$ in $\RR^d$, i.e., 
\begin{align}
\OL_{\lambda_0,\mu_0}[\bGam](\bx)= \delta_\mathbf{0}(\bx)\I_d, \qquad\forall\,\bx\in\RR^d,
\label{Gamma}
\end{align} 
where $\I_d\in\RR^{d\times d}$ is the identity matrix and $\delta_{\mathbf{0}}$ is the Dirac mass at $\mathbf{0}$. It is well known that (see, for instance, \cite[Lemma 6.2]{AK-Book}) 
\begin{align*}
\bGam (\bx)=
\begin{cases}
\ds\alpha\ln|\bx|\I_2-\frac{\beta}{|\bx|^2}\bx \bx^\top,  & \text{for}\;\; d=2,
\\ 
\ds-\frac{\alpha}{2|\bx|}\I_3-\frac{\beta}{2|\bx|^3}\bx \bx^\top,  &\text{for}\;\; d=3,
\end{cases}
\end{align*}
where the parameters $\alpha$ and $\beta$ are given by
\begin{equation}
\label{AlphaBeta}
\alpha:=\frac{\lambda_0+3\mu_0}{4\pi\mu_0(\lambda_0+2\mu_0)}
\quad\text{and}\quad
\beta:=\frac{\lambda_0+\mu_0}{4\pi\mu_0(\lambda_0+2\mu_0)}.
\end{equation}
The double layer potential associated with operator $\OL_{\lambda_0,\mu_0}$ in $\Omega$, hereinafter denoted  by $\mathcal{D}_\Omega$, is defined as
\begin{eqnarray*}
\mathcal{D}_\Omega[\bphi](\bx):=\int_{\partial \Omega}\frac{\partial}{\partial\bnu_\by}\bGam(\bx-\by)\cdot\bphi(\by)d\sigma(\by), \qquad \bx\in\RR^d\setminus\partial\Omega,
\end{eqnarray*}
for all  $\bphi\in L^2(\partial\Omega )^d$.  Remind that,   $\mathcal{D}_\Omega[\bphi]$ is a solution to $\OL_{\lambda_0,\mu_0}\big[\mathcal{D}_\Omega[\bphi]\big]=\mathbf{0}$ in $\RR^d\setminus\partial\Omega$ for all $\bphi\in L^2_\Psi(\partial\Omega)$ and the quantities $\mathcal{D}_\Omega[\bphi]\big|_{\pm}$ are well-defined for all $\bx\in\partial\Omega$. In fact, the following jump relations hold (see, e.g., \cite{Dahlberg})  
\begin{align}
\label{Djumps}
\ds\mathcal{D}_\Omega[\bphi]\Big|_{\pm}(\bx)=\left(\mp\frac{1}{2}\mathcal{I}+\mathcal{K}_\Omega\right)\bphi(\bx), \qquad{\rm a.e.}\quad\bx\in\partial \Omega,
\end{align}
where $\mathcal{I}:L^2(\partial\Omega)^d\to L^2(\partial\Omega)^d$ is the identity map and $\Kcal_\Omega$ is the so-called \emph{Neumann-Poincar\'e operator}, defined by 
\begin{equation*}
\Kcal_\Omega[\bphi](\bx) := {\rm p.v.} \ds\int_{\partial \Omega}\frac{\partial}{\partial\bnu_\by}\bGam (\bx-\by)\cdot\bphi(\by)d\sigma(\by),\quad{\rm a.e.}\quad\,\bx\in\partial\Omega,
\end{equation*}
for all $\bphi\in L^2(\partial\Omega)^d$. Here \emph{p.v.} stands for the Cauchy principle value. 

Let $\NN(\cdot,\by):\overline{\Omega}\to\RR^{d\times d}$, for a fixed $\by\in\overline{\Omega}$, be the Neumann function for the background domain $\Omega$ without any inclusion, i.e., the weak solution to  
\begin{equation}
\begin{cases}\label{Neumann2}
\OL_{\lambda_0,\mu_0}[\NN] (\bx,\by) =-\delta_\by(\bx)\I_d,   &  \text{for } \bx\in\Omega ,
\\ 
\ds\frac{\partial\NN}{\partial\bnu_\bx}(\bx,\by)=-\frac{1}{|\partial\Omega|}\I_d,  & \text{for } \bx\in\partial\Omega,
\end{cases}
\end{equation}
subject to the
%normalization 
condition $\NN(\cdot,\by)\in L^2_\Psi(\partial \Omega)$, for all $\by\in\overline{\Omega}$.
%$$
%\ds\int_{\partial\Omega}\NN(\bx,\by)\cdot\bpsi(\bx)d\sigma(\bx)=\mathbf{0}, \quad \forall\, \bpsi\in\Psi.
%$$

The following result from \cite[Lemma 6.16]{AK-Book}  is of great significance in the latter analysis. 
\begin{lem}\label{N-K-Gamma}
For all $\bx\in\partial\Omega$ and $\by\in\Omega$,
\begin{equation*} 
\left(-\ds\frac{1}{2}\mathcal{I}+ \Kcal_\Omega\right) [\NN(\cdot,\by)](\bx)=\bGam(\bx,\by)\quad\text{\rm modulo}\,\,\Psi.
%\label{NGam2}
\end{equation*}
\end{lem}

Lemma \ref{N-K-Gamma} indicates that the operator $(-\mathcal{I}/2+\Kcal_\Omega)$ \emph{filters} the effects of an imposed traction condition. It maps the solution of the elastostatic system in $\Omega$ with traction condition on $\partial \Omega$ to a solution of elastostatic system in $\RR^d$ subject to radiation conditions. The operator $(-\mathcal{I}/2+\Kcal_\Omega)$ will be coined as \emph{Calder\'{o}n preconditioner} in the sequel and will be effectively used to design an efficient reconstruction algorithm.

\subsection{Integral representation of perturbed displacement field}\label{ss:IntEq}

To facilitate the ensuing analysis, it is best to recall beforehand the Green's identities corresponding to elastostatic system. These identities can be easily derived using integration by parts and the divergence theorem (see, for instance, \cite[Sec. 9.1]{AK-Book}).
\begin{itemize}
\item If $\bv\in H^1(\Omega)^d$ and $\bw\in H^{3/2}(\Omega)^d$ then 
\begin{eqnarray}
\int_{\partial\Omega}\bv\cdot\frac{\partial \bw}{\partial\bnu} d\sigma = \int_\Omega\bv\cdot\OL_{\lambda_0,\mu_0}[\bw] d\bx+ \mathcal{Q}_\Omega^{\CC^0}(\bv,\bw),\label{GreenID1}
\end{eqnarray}
where the  quadratic (or so-called elastic energy) form $\mathcal{Q}_\Omega^{\CC^0}: H^{1}(\Omega)^d\times H^{1}(\Omega)^d\to \RR$ is defined by 
\begin{align}
\mathcal{Q}_\Omega^{\CC^0}(\bv,\bw):=
&\int_{\Omega}\mathcal{E}[\bv]:\CC^0:\mathcal{E}[\bw]d\bx
\nonumber
\\
=&\int_{\Omega}\Big[\lambda_0(\nabla\cdot\bv)(\nabla\cdot\bw)+2\mu_0(\mathcal{E}[\bv]:\mathcal{E}[\bw])\Big]d\bx.
\label{QuadraticForm}
\end{align}
Once again, the \emph{de facto} extension of $\mathcal{Q}^{\CC^0}_{\Omega}$ is used if one of the arguments is a matrix, in which case it would furnish a vector field.

\item If  $\bv,\bw\in H^{3/2}(\Omega)^d$ then the formula \eqref{GreenID1} leads to another Green's identity 
\begin{eqnarray}
\int_{\partial\Omega} \left(\bv\cdot\frac{\partial\bw}{\partial\bnu}-\bw\cdot\frac{\partial \bv}{\partial\bnu}\right)d\sigma(\bx)
=
\int_\Omega\left(\bv\cdot\OL_{\lambda_0,\mu_0}[\bw]-\bw\cdot\OL_{\lambda_0,\mu_0}[\bv]\right)d\bx.\label{GreenID2}
\end{eqnarray}
It is interesting to note that, by formula \eqref{GreenID2},  if $\bv$ is in addition a solution to $\OL_{\lambda_0,\mu_0}[\bv]=\mathbf{0}$ in $\Omega$ then   ${\partial\bv}/{\partial\bnu}\in L^2_\Psi(\partial\Omega)$.
\end{itemize}

The first important result of this section is the following integral representation of the perturbed displacement field. 
\begin{lem}
\label{LemIntEq1}
For all $\bx\in\overline{\Omega}$,  
\begin{align}
\bu_m(\bx)-\bU_m(\bx)=\int_{\cup_{n=1}^N D_n}\Big[
\big(\lambda_0-\lambda(\by)\big)
&\big(\nabla_\by\cdot\NN(\bx,\by)\big)\big(\nabla\cdot\bu_m(\by)\big)
\nonumber
\\
&
+2(\mu_0-\mu(\by))\mathcal{E}[\NN(\bx,\cdot)](\by):\mathcal{E}[\bu_m](\by)
\Big]d\by.
\label{IntEq1}
\end{align}
\end{lem}
\begin{proof}
Let us start with the  representation,
$$
\bU_m(\bx)=\int_{\partial\Omega}\NN(\bx,\by)\cdot\bg_m(\by)d\sigma(\by)=\int_{\partial\Omega}\NN(\bx,\by)\cdot\frac{\partial\bu_m}{\partial\bnu}(\by)d\sigma(\by), \quad  \bx\in\overline{\Omega}, 
$$
of the background solution $\bU_m$.
Note also that, by virtue of the imposed boundary condition on $\NN$ and the fact that $\bu_m\big|_{\partial\Omega}\in L^2_{\Psi}(\partial\Omega)$,   
$$
\int_{\partial\Omega}\frac{\partial}{\partial\bnu_\by}\NN(\bx,\by)\cdot\bu_m(\by)d\sigma(\by)=\mathbf{0},\quad\bx\in\overline{\Omega}.
$$
Therefore, the background field can be expressed as
\begin{align*}
\bU_m(\bx)=\int_{\partial\Omega}\left(\NN(\bx,\by)\cdot\frac{\partial\bu_m}{\partial\bnu}(\by)-\frac{\partial}{\partial\bnu_\by}\NN(\bx,\by)\cdot\bu_m(\by)\right)d\sigma(\by), \quad \bx\in\overline{\Omega}.
\end{align*}
Consequently, a simple application of the Green's formula \eqref{GreenID2} over $\Omega\setminus \cup_{n=1}^N \overline{D_n}$ furnishes 
\begin{align}
\bU_m(\bx)=&-\int_{\Omega\setminus\cup_{n=1}^N \overline{D_n}} \Big(\OL_{\lambda_0,\mu_0}[\NN](\bx,\by)\cdot\bu_m(\by)-\NN(\bx,\by)\cdot\OL_{\lambda_0,\mu_0}[\bu_m](\by)\Big)d\by
\nonumber
\\
&+\sum_{n=1}^N\int_{\partial D_n}\left(\NN(\bx,\by)\cdot\frac{\partial\bu_m}{\partial\bnu}(\by)\Big|_{+}-\frac{\partial}{\partial\bnu_\by}\NN(\bx,\by)\cdot\bu_m(\by)\Big|_+\right)d\sigma(\by),\quad\bx\in\overline{\Omega}.\label{eq:1}
\end{align}

Remark that the second term on the right hand side (RHS) of \eqref{eq:1} is identically zero thanks to the first equation in \eqref{Prob_Trans2}.  Moreover,  by invoking the transmission conditions on the displacement $\bu_m$  and its surface traction on the boundaries $\partial D_n$ in \eqref{Prob_Trans2},  one obtains
\begin{align*}
\bU_m(\bx)=&-\int_{\Omega\setminus \cup_{n=1}^N \overline{D_n}}\OL_{\lambda_0,\mu_0}[\NN](\bx,\by)\cdot\bu_m(\by)d\by
\nonumber
\\
&+\sum_{n=1}^N\int_{\partial D_n}\left(\NN(\bx,\by)\cdot\frac{\partial\bu_m}{\partial\widetilde{\bnu}}(\by)\Big|_{-}-\frac{\partial}{\partial\bnu_\by}\NN(\bx,\by)\cdot\bu_m(\by)\Big|_-\right)d\sigma(\by), \quad \bx\in\overline{\Omega}.
\end{align*}
By applying the Green's identity \eqref{GreenID2} once again, but this time over $\cup_{n=1}^N D_n$, one gets
\begin{align}
\bU_m(\bx)=&
-\int_{\Omega}\OL_{\lambda_0,\mu_0}[\NN](\bx,\by)\cdot\bu_m(\by)d\by
+\int_{\cup_{n=1}^N D_n}\NN(\bx,\by)\cdot\OL_{\lambda_0,\mu_0}[\bu_m](\by)d\by
\nonumber
\\
&+\sum_{n=1}^N\int_{\partial D_n}\NN(\bx,\by)\cdot\left(\frac{\partial\bu_m}{\partial\widetilde{\bnu}}(\by)\Big|_{-}-\frac{\partial\bu_m}{\partial\bnu}(\by)\Big|_{-}\right)d\sigma(\by), \quad \bx\in\overline{\Omega}.
\label{eq:Intermid1}
\end{align}
Thanks to the first equation in \eqref{Neumann2} for the Neumann solution, one can easily identify the first term on the RHS of \eqref{eq:Intermid1} as  $\bu_m(\bx)$. Therefore,  
\begin{align}
\bU_m(\bx)-\bu_m(\bx)=&
\int_{\cup_{n=1}^N D_n}\NN(\bx,\by)\cdot\OL_{\lambda_0,\mu_0}[\bu_m](\by)d\by
\nonumber
\\
&
+\sum_{n=1}^N\int_{\partial D_n}\NN(\bx,\by)\cdot\left(\frac{\partial\bu_m}{\partial\widetilde{\bnu}}(\by)\Big|_{-}-\frac{\partial\bu_m}{\partial\bnu}(\by)\Big|_{-}\right)d\sigma(\by), \quad \bx\in\overline{\Omega}.
\label{Eq2.4}
\end{align}

On the other hand, by using Green's identity  \eqref{GreenID1} over $\cup_{n=1}^N D_n$, it is found that
\begin{align*}
\int_{\cup_{n=1}^N D_n}\NN  (\bx,\by) 
\cdot\OL_{\lambda_0,\mu_0}[\bu_m](\by)d\by
=& 
\sum_{n=1}^N\int_{\partial D_n}\NN(\bx,\by)\cdot\frac{\partial\bu_m}{\partial\bnu}(\by)\Big|_{-}d\sigma(\by)
\\
&-\mathcal{Q}_{\cup_{n=1}^N D_n}^{\CC^0}\left(\NN(\bx,\cdot),\bu_m\right), \quad \bx\in\overline{\Omega}.
\end{align*}
Consequently, a few fairly easy manipulations lead us to 
\begin{align*}
\int_{\cup_{n=1}^N D_n}\NN (\bx,\by)\cdot\OL_{\lambda_0,\mu_0}[\bu_m](\by)d\by
=& 
\sum_{n=1}^N\int_{\partial D_n}\NN(\bx,\by)\cdot\left(\frac{\partial\bu_m}{\partial\bnu}(\by)\Big|_{-}-\frac{\partial\bu_m}{\partial\widetilde{\bnu}}(\by)\Big|_{-}\right)d\sigma(\by)
\\
&
-\mathcal{Q}_{\cup_{n=1}^N D_n}^{\CC^0-\CC}(\NN(\bx,\cdot),\bu_m) 
\\
&
+\int_{\cup_{n=1}^N D_n}\NN(\bx,\by)\cdot\OL_{\lambda,\mu}[\bu_m](\by)d\by, \quad \bx\in\overline{\Omega}, 
\end{align*}
wherein the last term vanishes thanks to the second equation in \eqref{Prob_Trans2}. Therefore,  
\begin{align}
\int_{\cup_{n=1}^N D_n}
&\NN (\bx,\by)\cdot\OL_{\lambda_0,\mu_0}[\bu_m](\by)d\by
+\sum_{n=1}^N\int_{\partial D_n}\NN(\bx,\by)\cdot\left(\frac{\partial\bu_m}{\partial\widetilde{\bnu}}(\by)\Big|_{-}-\frac{\partial\bu_m}{\partial\bnu}(\by)\Big|_{-}\right)d\sigma(\by)
\nonumber
\\
=&
-\mathcal{Q}_{\cup_{n=1}^N D_n}^{\CC^0-\CC}(\NN(\bx,\cdot),\bu_m),
\quad \forall \bx\in\overline{\Omega}.\label{eq:3}
\end{align}
Finally, combining \eqref{Eq2.4} and \eqref{eq:3}, it can be seen that 
\begin{align*}
\bu_m(\bx)-\bU_m(\bx)=&\mathcal{Q}_{\cup_{n=1}^N D_n}^{\CC^0-\CC}(\NN(\bx,\cdot),\bu_m), \quad \bx\in\overline{\Omega}.
\end{align*}
This leads to the conclusion together with the definition \eqref{QuadraticForm} of the quadratic form.
\end{proof}

The integral equation \eqref{IntEq1} provides an exact expression for the perturbations in the displacement field. However, it is important to note that the Neumann function $\NN$ does not admit an exact explicit expression (except in a few trivial cases of simple domains such as disks and balls) in contrast to the Kelvin matrix $\bGam$ which is easily accessible. Therefore, although it is exact and valid for all $\bx\in\overline{\Omega}$, the practical utility of the relation \eqref{IntEq1} is restricted.  Nevertheless, a \emph{pre-processing} of the  data using Calder\'on preconditioner is particularly felicitous. In fact, by preconditioning integral equation  \eqref{IntEq1} using $(-\mathcal{I}/2+\Kcal_\Omega)$  and invoking the Lemma \ref{N-K-Gamma},  one can easily obtain an equivalent representation of the \emph{filtered measurements} in terms of the Kelvin matrix. The following result is readily proved thanks to Lemmas \ref{N-K-Gamma}--\ref{LemIntEq1}.
\begin{lem} 
\label{LemIntEq2}
For a.e. $\bx\in\partial\Omega$, 
\begin{align}
\left(-\frac{1}{2}\mathcal{I}+\Kcal_\Omega\right)\left[\bu_m-\bU_m\right](\bx)
=&
\int_{\cup_{n=1}^N D_n}\Big[\big(\lambda_0-\lambda(\by)\big)\big(\nabla_\by\cdot\bGam(\bx,\by)\big)\big(\nabla\cdot\bu_m(\by)\big)
\nonumber
\\
&
+2(\mu_0-\mu(\by))\mathcal{E}[\bGam(\bx,\cdot)](\by):\mathcal{E}[\bu_m](\by)\Big]d\by.  
\label{IntEqGam1}
\end{align}
\end{lem}
It is emphasized that, however the integral equation \eqref{IntEq1} is valid for all $\bx\in\overline{\Omega}$, its preconditioned counterpart \eqref{IntEqGam1} is only valid for a.e. $\bx\in\partial\Omega$. In order to derive an alternative integral equation that is valid for all  $\bx\in\Omega$, one can exploit the double layer potential $\mathcal{D}_\Omega$ in view of the jump relations \eqref{Djumps}. In fact, the following result holds.

\begin{lem}
\label{LemIntEq3}
 For all $\bx\in\Omega$,
\begin{align}
\bu_m(\bx)-\bU_m(\bx)=& \mathcal{D}_\Omega\left[(\bu_m-\bU_m)\big|_{\partial\Omega}\right](\bx)-\int_{\cup_{n=1}^N D_n}\Big[\big(\lambda_0-\lambda(\by)\big)\big(\nabla_\by\cdot\bGam(\bx,\by)\big)\big(\nabla\cdot\bu_m(\by)\big)
\nonumber
\\
&
+2(\mu_0-\mu(\by))\mathcal{E}[\bGam(\bx,\cdot)](\by):\mathcal{E}[\bu_m](\by)\Big]d\by.
\label{IntEqGam2}
\end{align}
\end{lem}

\begin{proof}
Note that, for all $\bx\in \Omega$, 
\begin{align}
\mathcal{D}_\Omega\left[(\bu_m-\bU_m)\big|_{\partial\Omega}\right](\bx)
= & 
\ds\int_{\partial\Omega} \frac{\partial\bGam}{\partial\bnu_\by}(\bx,\by)\cdot\Big(\bu_m(\by)-\bU_m(\by)\Big)d\sigma(\by)
\nonumber
\\
\nm
=&\ds\int_{\partial\Omega} \left(\frac{\partial\bGam}{\partial\bnu_\by}(\bx,\by)\cdot\bu_m(\by)-\bGam(\bx,\by)\cdot\frac{\partial\bu_m(\by)}{\partial\bnu}\right) d\sigma(\by)
\nonumber
\\
&
\ds
-\int_{\partial\Omega} \left(\frac{\partial\bGam}{\partial\bnu_\by}(\bx,\by)\cdot\bU_m(\by)-\bGam(\bx,\by)\cdot\frac{\partial\bU_m(\by)}{\partial\bnu}\right) d\sigma(\by),
\label{D1} 
\end{align}
where the fact that $\bu_m$ and $\bU_m$ satisfy same traction boundary conditions on $\partial\Omega$ is exploited. It can be easily verified, using the Green's identity \eqref{GreenID2}, that
\begin{align}
\int_{\partial\Omega} \left(\frac{\partial\bGam}{\partial\bnu_\by}(\bx,\by)\cdot\bU_m(\by)-\bGam(\bx,\by)\cdot\frac{\partial\bU_m(\by)}{\partial\bnu}\right) d\sigma(\by) = \bU_m(\bx),\qquad\bx\in \Omega,
\label{D2}
\end{align}
since $\Omega$ is a $\mathcal{C}^2-$ domain.
On the other hand, using \eqref{GreenID2} over $\Omega\setminus\cup_{n=1}^N\overline{D_n}$, one obtains
\begin{align*}
\mathbf{T}(\bx):=&\int_{\partial\Omega} \left(\frac{\partial\bGam}{\partial\bnu_\by}(\bx,\by)\cdot\bu_m(\by)-\bGam(\bx,\by)\cdot\frac{\partial\bu_m(\by)}{\partial\bnu}\right) d\sigma(\by) 
\\
= &
\int_{\Omega\setminus \cup_{n=1}^N \overline{D_n}}\Big(\OL_{\lambda_0,\mu_0}[\bGam](\bx,\by)\cdot\bu_m(\by)- \bGam(\bx,\by)\cdot\OL_{\lambda_0,\mu_0}[\bu_m](\by)\Big)d\by
\\
&-\sum_{n=1}^N\int_{\partial D_n}\left(\bGam(\bx,\by)\cdot\frac{\partial\bu_m}{\partial\bnu}(\by)\Big|_{+}-\frac{\partial\bGam}{\partial\bnu_\by}(\bx,\by)\cdot\bu_m(\by)\Big|_+\right)d\sigma(\by),\quad\bx\in {\Omega}.
\end{align*}
By making use of the system \eqref{Prob_Trans2}, one gets
\begin{align*}
\mathbf{T}(\bx)=& \int_{\Omega\setminus \cup_{n=1}^N \overline{D_n}}\OL_{\lambda_0,\mu_0}[\bGam](\bx,\by)\cdot\bu_m(\by)d\by
\nonumber
\\
&-\sum_{n=1}^N\int_{\partial D_n}\left(\bGam(\bx,\by)\cdot\frac{\partial\bu_m}{\partial\widetilde{\bnu}}(\by)\Big|_{-}-\frac{\partial\bGam}{\partial\bnu_\by}(\bx,\by)\cdot\bu_m(\by)\Big|_-\right)d\sigma(\by), \quad \bx\in {\Omega}.
\end{align*}
Subsequently, by applying Green's identity \eqref{GreenID2} over $\cup_{n=1}^N D_n$, one arrives at
\begin{align}
\mathbf{T}(\bx)=& \int_{\Omega}\OL_{\lambda_0,\mu_0}[\bGam](\bx,\by)\cdot\bu_m(\by)d\by
-\int_{\cup_{n=1}^N D_n}\bGam(\bx,\by)\cdot\OL_{\lambda_0,\mu_0}[\bu_m](\by)d\by
\nonumber
\\
&-\sum_{n=1}^N\int_{\partial D_n}\bGam(\bx,\by)\cdot\left(\frac{\partial\bu_m}{\partial\widetilde{\bnu}}(\by)\Big|_{-}-\frac{\partial\bu_m}{\partial\bnu}(\by)\Big|_{-}\right)d\sigma(\by), \quad \bx\in {\Omega}.
\label{D4}
\end{align}
One can easily identify the first term on the RHS of \eqref{D4} as  $\bu_m(\bx)$, thanks to \eqref{Gamma}. Moreover, fairly simple arguments, similar to those in the proof of Lemma \ref{LemIntEq1}, lead us to 
\begin{align}
\int_{\cup_{n=1}^N D_n}
&\bGam (\bx,\by)\cdot\OL_{\lambda_0,\mu_0}[\bu_m](\by)d\by
+\sum_{n=1}^N\int_{\partial D_n}\bGam(\bx,\by)\cdot\left(\frac{\partial\bu_m}{\partial\widetilde{\bnu}}(\by)\Big|_{-}-\frac{\partial\bu_m}{\partial\bnu}(\by)\Big|_{-}\right)d\sigma(\by)
\nonumber
\\
=&
-\mathcal{Q}_{\cup_{n=1}^N D_n}^{\CC^0-\CC}(\bGam(\bx,\cdot),\bu_m),
\quad \forall \bx\in{\Omega}.\label{D6}
\end{align}
Therefore, by using the expression \eqref{D6} in \eqref{D4}, it is found that
\begin{align}
\mathbf{T}(\bx)= \bu_m(\bx)+ \mathcal{Q}_{\cup_{n=1}^N D_n}^{\CC^0-\CC}(\bGam(\bx,\cdot),\bu_m),
\quad \forall \bx\in{\Omega}.\label{D3}
\end{align} 
Finally, using \eqref{D2} and \eqref{D3} in \eqref{D1}, one arrives at 
\begin{align*}
\mathcal{D}_\Omega\left[\left(\bu_m-\bU_m\right)\big|_{\partial\Omega} \right](\bx) = \bu_m(\bx)-\bU_m(\bx)+ \mathcal{Q}_{\cup_{n=1}^N D_n}^{\CC^0-\CC}(\bGam(\bx,\cdot),\bu_m), \qquad\forall \bx\in\Omega,
\end{align*}
which renders the expression \eqref{IntEqGam2} by virtue of the definition \eqref{QuadraticForm}.
\end{proof}

The integral equations \eqref{IntEqGam1} and \eqref{IntEqGam2} are the key components of  our proposed algorithm.  
For completeness, the analytic expressions for different kernels involved in \eqref{IntEqGam1} and \eqref{IntEqGam2} are provided in Appendix \ref{Append}.
However, the presence of $\bu_m$ on their right hand sides is problematic since $\bu_m$  is only available  on $\partial \Omega$ \emph{\`{a} priori}. 
Different approaches based on the linearized versions of \eqref{IntEq1} and \eqref{IntEqGam1}  or on Born and Rytov type approximations of  $\bu_m$ are available in the literature.  Unfortunately, these simplifications are not always valid and become too strong to allow an accurate reconstruction.  However, as will be shown in Section \ref{s:Recon}, a joint sparsity based reformulation of \eqref{IntEqGam1} in terms of unknown densities is possible if the inclusions, $D_1,\cdots, D_N$, are compactly supported and \emph{sufficiently localized} inside $\Omega$, i.e., the support set $\cup_{n=1}^n D_n$ is \emph{sparse} in $\Omega$. Consequently, linearization or approximations can be avoided. It is also elaborated how this allows us to recover support set $\cup_{n=1}^n D_n$ and the Lam\'{e} parameters $(\lambda_n,\mu_n)$ without any linearization or iterative update when multiple measurements are available. 
\section{Joint sparse reconstruction framework}\label{s:Recon}

In this section, the inverse problem for spatial localization of inclusions is recast to a joint sparse recovery problem and it is shown that the problem for quantitative evaluation of Lam\'{e} parameters becomes linear in sought parameters  if the underlying inclusions are sparsely embedded in elastic formation.
In particular, the integral representation \eqref{IntEqGam1}  for multiple perturbed displacement fields, corresponding to different applied boundary forces $\bg_1,\cdots,\bg_M$, will be reformulated in terms of jointly sparse densities in Section \ref{ss:JSP}. This will allow us to invoke compressed sensing algorithms for so-called \emph{multiple measurement vector} problems for sparse signal recovery (see, for instance, \cite{cotter2005ssl, Kim2010CMUSIC, tropp2006ass2, wipf2007ebs}), thereby furnishing the unknown densities inside the support set $\cup_{n=1}^N D_n$.   Furthermore, in Section \ref{ss:Parameters},  it is established using the recovered densities together with \eqref{IntEqGam2} that the  accurate estimation of the displacement and the strain inside the inclusions is possible. Consequently, the inverse problem for parameter reconstruction becomes linear. The issues related to the discretization of these formulations and the imaging procedure using finite discrete measurements will be discussed in Section \ref{ss:discrete}.

\subsection{Integral formulation using jointly sparse densities and support identification}\label{ss:JSP}

Let us first investigate the integral formulation \eqref{IntEqGam1} subject to multiple boundary forces.  Towards this end, let $\mathbf{A}^i:\RR^d\times\RR^d\to\RR^{1\times 1}$, $\mathbf{B}^i:\RR^d\times\RR^d\to\RR^{d\times d}$, $\mathbf{X}_m^1:\RR^d\to\RR^{1\times 1}$, and $\mathbf{X}_m^2:\RR^d\to\RR^{d\times d}$, for all $i=1,\cdots,d$ and $m=1,\cdots, M$, be defined by 
\begin{align*}
&\mathbf{A}^i(\bx,\by):= 
\Big([\nabla\cdot\bGam(\bx,\by)]_i\Big), 
%\quad \bx\neq\by,
\qquad
\mathbf{B}^i(\bx,\by):= 
{\rm slice}_1\Big(i, \mathcal{E}[\bGam]
(\bx,\by)\Big), \quad \bx\neq\by,  
\\\nm
&\mathbf{X}_m^1(\by):= \left[\lambda_0-\lambda(\by)\right]\Big(\nabla\cdot\bu_m(\by)\Big),  
\qquad
\mathbf{X}_m^2(\by):= 
2\left[\mu_0-\mu(\by)\right]\mathcal{E}[\bu_m](\by).
\end{align*}
With these definitions at hand,  \eqref{IntEqGam1} can be rewritten as
\begin{align}
\Big(-\frac{1}{2}&\mathcal{I}+\Kcal_\Omega\Big)\left[\bu_m-\bU_m\right](\bx)
=
\int_{\cup_{n=1}^N D_n}\sum_{i=1}^d\Big[\mathbf{A}^i(\bx,\by):\mathbf{X}^1_m(\by)+\mathbf{B}^i(\bx,\by):\mathbf{X}^2_m(\by)\Big]\be_i\,d\by
\nonumber
\\
=&
\int_{\cup_{n=1}^N D_n}\sum_{i=1}^d\Big[{\rm vec}\left(\mathbf{A}^i(\bx,\by)\right)^\top{\rm vec}\left(\mathbf{X}^1_m(\by)\right)+{\rm vec}\left(\mathbf{B}^i(\bx,\by)\right)^\top{\rm vec}\left(\mathbf{X}^2_m(\by)\right)\Big]\be_i\,d\by, 
\label{eq:ast} 
\end{align}
for all $\bx\in\partial\Omega$. Let us also introduce $\mathbf{X}_{m}:\cup_{n=1}^N D\to\RR^{d^2+1}$,  
 $\bLambda:\cup_{n=1}^N D \times \cup_{n=1}^N D\to\RR^{d\times (d^2+1)}$, and $\mathbf{Y}_{m}:\partial\Omega\to\RR^{d}$ by
\begin{align*}
&\mathbf{X}_{m}(\by):=
\begin{pmatrix}
{\rm vec}\left(\mathbf{X}_m^1(\by)\right)
\\
{\rm vec}\left(\mathbf{X}_m^2(\by)\right)
\end{pmatrix},
\quad
\bLambda(\bx,\by)
:=
\begin{pmatrix}
{\rm vec}\left(\mathbf{A}^1(\bx,\by)\right)^\top  
&
& 
{\rm vec}\left(\mathbf{B}^1(\bx,\by)\right)^\top
\\
\vdots&&\vdots
\\
{\rm vec}\left(\mathbf{A}^d(\bx,\by)\right)^\top  &
& 
{\rm vec}\left(\mathbf{B}^d(\bx,\by)\right)^\top
\end{pmatrix},
\\
&\mathbf{Y}_{m}(\bx):=
\left(-\frac{1}{2}\mathcal{I}+\Kcal_\Omega\right)\left[\bu_m-\bU_m\right](\bx),
\end{align*}
so that, from \eqref{eq:ast},
\begin{align}
\mathbf{Y}_m(\bx)
=
\int_{\cup_{n=1}^N D_n} \bLambda(\bx,\by)\mathbf{X}_{m}(\by)d\by, \quad \bx\in\partial\Omega.
\label{IntEqMatrix1}
\end{align}
Remark that the integral in \eqref{IntEqMatrix1} has to be evaluated over the unknown support of the inclusions. In order to furnish an integral equation that does not require \emph{\`a priori} information of the unknown support of the inclusions, the function  $\mathbf{X}_m$ is simply extended by zero outside $\cup_{n=1}^N D_n$, i.e., its extension, $\widetilde{\mathbf{X}}_m$, defined by 
\begin{align*}
\widetilde{\mathbf{X}}_m:= 
\begin{cases}
\mathbf{X}_{m}(\by),  & \text{for } \by\in\cup_{n=1}^N D_n, 
\\
\mathbf{0},  & \text{for }\by\in\Omega\setminus\cup_{n=1}^N \overline{D_n},
\end{cases}
\end{align*}
is considered so that 
\begin{align}
\mathbf{Y}_m(\bx)
=
\int_{\Omega} \bLambda(\bx,\by)\widetilde{\mathbf{X}}_{m}(\by)d\by, \quad \bx\in\partial\Omega.
\label{IntEqMatrix2}
\end{align}

It is very interesting to note that the inclusions are compactly embedded well inside the background domain $\Omega$  (thanks to assumption H1) and are located at fixed positions despite of the different applied boundary forces $\bg_1,\cdots,\bg_M$. Moreover, the density $\widetilde{\mathbf{X}}_m$ varies only at the support of the inclusions, $D_1,\cdots, D_N$, for each excitation but is zero elsewhere  independent of the measurement data $\mathbf{Y}_m$. Therefore, assuming sparsity for the support set $\cup_{n=1}^N D_N$ in $\Omega$, the problem of inclusion detection from integral form \eqref{IntEqMatrix2} can be regarded as a joint sparse recovery problem of $\widetilde{\mathbf{X}}_{m}(\by), m=1,\cdots,M$, which has been extensively investigated in compressed sensing literature \cite{chen2006trs,Kim2010CMUSIC}.  A detailed implementation including the discretization of integral form \eqref{IntEqMatrix2} for joint sparse recovery will be discussed at a later stage. 
Once  the density  $\widetilde{\mathbf{X}}_{m}(\by), m=1,\cdots, M,$ is reconstructed using a joint sparse recovery algorithm, the support of the inclusions can be easily identified by  investigating the magnitudes of  $\widetilde{\mathbf{X}}_{m}(\by)$.

\subsection{Recovery of constitutive parameters}\label{ss:Parameters}

In the sequel,  the notation $\hat{D}$ is adopted for the entire reconstructed support $\cup_{n=1}^N D_n$ by virtue of joint sparse recovery step. 
 Similarly,  $\hat{\mathbf{X}}_{m}(\by)$ (and, accordingly, $\hat{\mathbf{X}}^1_m(\by), \hat {\mathbf{X}}^2_m(\by)$) will denote the estimated densities.
The second step of the proposed algorithm dealing with the parameter evaluation is based on the integral equation \eqref{IntEqGam2}.  Precisely, first the total displacement field  is estimated for all $\bx\in\hat{D}$ using the recursive relationship
\begin{align*}
%\label{eq:ukxUkxInt2}
\widehat{\bu}_m(\bx)=\bU_m(\bx)+\mathcal{D}_\Omega
&
\left[(\bu_m-\bU_m)\big|_{\partial\Omega}\right](\bx)
\nonumber
\\
&-
\int_{\hat{D}}\sum_{i=1}^d\Big[\mathbf{A}^i(\bx,\by):\hat{\mathbf{X}}^1_m(\by)+\mathbf{B}^i(\bx,\by):\hat{\mathbf{X}}^2_m(\by)\Big]\be_i\,d\by,
\end{align*}
where  $\widehat{\bu}_m$ is the calculated total field over $\hat{D}$ at this step and all the terms on the RHS are obtained either from the measurements or the first step. Then, the quantities $\nabla\cdot\widehat{\bu}_m$ and $\mathcal{E}\left(\widehat{\bu}_m\right)$ can also be computed from $\widehat{\bu}_m$ for all $\bx\in\hat{D}$. Finally, the integral formulation \eqref{IntEqGam1} or its matrix form  \eqref{IntEqMatrix1} is invoked once again to formulate another problem with slightly modified sensing matrix and new unknowns, $(\lambda_0-\lambda)$ and $(\mu_0-\mu)$. Indeed, from  \eqref{IntEqMatrix1},  
\begin{align}
\mathbf{Y}_m(\bx)
&=  \int_{\hat{D}} \bLambda(\bx,\by) \hat{\mathbf{X}}_m(\by)d\by
=
\int_{\hat{D}} \widetilde{\bLambda}_m(\bx,\by)\mathbf{Z}(\by)d\by,
\label{IntEqMatrix3}
\end{align}
where $\widetilde{\bLambda}_m : \cup_{n=1}^{N} D_n \times \cup_{n=1}^{N} D_n \to \RR^{d\times (d^2+1)}$ and $\mathbf{Z}:\cup_{n=1}^{N} D_n \to \RR^{d^2+1}$ are defined by 
\begin{align*}
&\widetilde{\bLambda}_m(\bx,\by) := 
\bLambda(\bx,\by){\rm diag}\Big(\nabla\cdot\hat{\bu}_m(\by), {\rm vec} \big(2\mathcal{E}(\hat{\bu}_m)(\by)\big)^\top\Big),
\\
&\mathbf{Z}(\by) :=  
\begin{pmatrix}
\lambda_0-\lambda(\by), 
&  
\mu_0-\mu(\by),
&
\cdots,
&
\mu_0-\mu(\by)
\end{pmatrix} ^\top.
\end{align*}

It is emphasized that the problem \eqref{IntEqMatrix3} is {\em linear} for $\mathbf{Z}$ since the support set $\hat{D}$ and  the modified sensing matrix $\widetilde{\bLambda}_m$ on it are completely known. Thus,   no linearization or iterative update is required to solve \eqref{IntEqMatrix3}. Note that, the second step of inverse problem is expected to be efficient and less ill-posed due to the knowledge of the estimated position of anomalies.

\subsection{Joint sparsity algorithm in discrete setting}\label{ss:discrete}

Let us now explain a  procedure to use the findings of Sections \ref{ss:JSP} and \ref{ss:Parameters} in a discrete setting. In doing so, the first critical step is the  Calder\'{o}n preconditioning of the discrete data obtained on a finite number of boundary points using the operator $(-\mathcal{I}/2+\Kcal_\Omega)$.  Once the measured data is filtered, the integral formulations \eqref{IntEqMatrix2} and \eqref{IntEqMatrix3} can be descritized for implementation of the joint sparsity algorithm using discrete measurements. 

\subsubsection{Calder\'{o}n preconditioning and discrete data interpolation}\label{ss:Calderon}

Although the left-hand side of the integral formulations \eqref{IntEqMatrix2} and \eqref{IntEqMatrix3} are the filtered outputs of $\bu_m(\bx)-\bU_m(\bx)$ using Calder\'{o}n preconditioning,
the actual measurement in real experiments are the discrete samples of $\bu_m(\bx)-\bU_m(\bx)$  instead of their filtered outputs.
This can be problematic because it can be assumed that the filtered outputs can  be calculated only if the measurements $\bu_m(\bx)-\bU_m(\bx)$ are available at all points along the boundary $\partial \Omega$, taking into account the continuous nature of the  Calder\'{o}n preconditioning operator.

However, this problem can be easily alleviated by interpolating the discrete measurement of $\bu_m(\bx)-\bU_m(\bx)$ along $\partial \Omega$ and then applying
the Calder\'{o}n preconditioning.  For this purpose, in particular high-order spline interpolation are used. This may result in potential interpolation errors in the filtered outputs  used for the integral formulations \eqref{IntEqMatrix2} and \eqref{IntEqMatrix3}.
However, due to the stability of the compressed sensing reconstruction, this does not involve any significant errors in the results of the final reconstruction,  as will be shown in the experimental section.

\subsubsection{Step one: Descritization of \eqref{IntEqMatrix2}}\label{sec:step1}

For numerical implementation,  let us assume that $\widetilde{\mathbf{X}}_{m}$ is approximated by either piece wise constant functions or splines as 
\begin{align*}
\left[\mathbf{\widetilde{X}}_{m}(\by)\right]_q=
\ds\sum_{\ell=1}^L\left[\mathbf{\widetilde{X}}_{m}(\by_{\ell})\right]_q\;\varphi_q\left(\by,\by_{\ell}\right),
\quad \forall q\in\{1,\cdots, d^2+1\},\,\,\by\in\Omega,
%\label{JmAprox}
\end{align*}
where $\{\by_{\ell}\}_{\ell=1}^L$, for some $L\in\mathbb{N}$, are the finite sampling points of $\Omega$ and $\varphi_q\left(\by,\by_{\ell}\right)$ is the basis function for the $q-$th coordinate with $q\in\{1,\cdots,d^2+1\}$.  

Using measurement points $\{\bx_{r}\}_{r=1}^R\subset\partial\Omega$  and  sampling points $\{\by_{\ell}\}_{\ell=1}^L \in\Omega$, let us introduce the unknown density $\mathfrak{X}\in\RR^{(d^2+1)L\times M}$, the measurement matrix $\mathfrak{Y}\in\RR^{dR\times M}$ and the sensing matrix $\mathbf{\Pi}\in\RR^{dR\times (d^2+1)L}$ by 
\begin{align*}%\label{eq:JSMformulation}
\mathfrak{X}:=
\begin{pmatrix}
\mathfrak{X}_1
\\
\vdots
\\
\mathfrak{X}_{d^2+1}
\end{pmatrix},
\quad
\mathfrak{Y}:=
\begin{pmatrix}
\mathfrak{Y}_1
\\
\vdots
\\
\mathfrak{Y}_{d}
\end{pmatrix}, 
\quad\text{and}\quad
\mathbf{\Pi}:=
\begin{pmatrix}
\mathbf{\Pi}_{11} & \cdots &\mathbf{\Pi}_{1\,(d^2+1)}
\\
\vdots & \ddots & \vdots
\\
\mathbf{\Pi}_{d1} & \cdots &\mathbf{\Pi}_{d\,(d^2+1)}
\end{pmatrix},
\end{align*}
where the element matrices  $\mathfrak{X}_q\in\RR^{L\times M}$, $\mathfrak{Y}_p\in\RR^{R\times M}$ and $\mathbf{\Pi}_{pq}\in\RR^{R\times L}$ are defined by
\begin{align*}
\left[\mathfrak{X}_q\right]_{\ell m}:=&\left[\widetilde{\mathbf{X}}_m(\by_\ell)\right]_q,
\,\,\,
\left[\mathfrak{Y}_p\right]_{rm}:=&\left[\mathbf{Y}_m(\bx_r)\right]_p,
\,\,\,\text{and}\,\,\,
[\mathbf{\Pi}_{pq}]_{r\ell} :=&\int_{\Omega}\left[\bLambda\left(\bx_r,\bz\right)\right]_{pq}\varphi_q\left(\bz,\by_\ell\right)d\bz,
\end{align*}
for all $p\in\{1,\cdots, d\}$, $q\in\{1,\cdots,d^2+1\}$, $r\in\{1,\cdots,R\}$ and $\ell\in\{1,\cdots, L\}$.
The aforementioned discretization and definitions render the system of linear equations 
\begin{eqnarray}
\mathfrak{Y}=\mathbf{\Pi}\mathfrak{X}.\label{eq:DisSys}
\end{eqnarray}

The following remarks are in order. 
Firstly, there are usually more sampling points $\by_{\ell}\in\Omega$ than the measurement points $\bx_{r}\in\partial\Omega$ in practice. 
Therefore, the linear system \eqref{eq:DisSys} is practically very under-determined, i.e., $dR\ll (d^2+1)L$. 
Therefore, there is no uniqueness of the solution without assuming any prior knowledge.
Secondly, the unknown density matrix $\mathfrak{X}$ is sparse thanks to its construction and the assumption of the sparsity of the support set $\cup_{n=1}^N D_n$ in $\Omega$.  In fact, in the aforementioned discrete setup the non-zero rows of $\mathfrak{X}$ are those that correspond to the locations $\by_{\ell}\in \cup_{n=1}^N D_n$. 
These two observations naturally lead us to exploit the joint sparsity as a prior information.
 As will be discussed at a later stage, there are several joint sparse recovery algorithms to uniquely solve system \eqref{eq:DisSys}. 
Therefore, by solving the linear system using any one of those algorithms, the unknown density $\mathfrak{X}$ can be recovered which in turn gives access to the support of the inclusions and the perturbed displacement field inside the support of the inclusions.

\subsubsection{Step two: Discretization of  \eqref{IntEqMatrix3}}

The formulation \eqref{IntEqMatrix3} renders another under-determined system of linear equations  in the discrete setting and $\mathbf{Z}(\by_\ell)$, for all $\by_\ell\in \hat D$, can be obtained by solving another constraint optimization problem. Towards this end, let $\{\hat{\by}_\ell\}_{\ell=1}^{\tilde{L}}$ be the collection of sampling points that belong to the support set $\hat{D}$ with $\widetilde{L} \ll L$. Then, by fairly easy manipulations similar to those in Section  \ref{sec:step1}, the descretized version of \eqref{IntEqMatrix3} is obtained as 
\begin{align}
\widetilde{\mathbf{Y}}=\widetilde{\mathbf{\Pi}}\widetilde{\mathbf{Z}},
\label{eq:DisSys2}
\end{align}
with 
\begin{align*}
\widetilde{\mathbf{Y}}:=
\begin{pmatrix}
\widetilde{\mathbf{Y}}_1
\\
\widetilde{\mathbf{Y}}_2
\\
\vdots
\\
\widetilde{\mathbf{Y}}_M
\end{pmatrix}\in\RR^{MdR},
\quad
\widetilde{\mathbf{Z}}:=
\begin{pmatrix}
\widetilde{\mathbf{Z}}_1
\\
\widetilde{\mathbf{Z}}_2
\\
\vdots
\\
\widetilde{\mathbf{Z}}_{d^2+1}
\end{pmatrix} \in\RR^{(d^2+1)\widetilde{L}},
\quad\text{and}\quad
\widetilde{\mathbf{\Pi}}:=
\begin{pmatrix}
\widetilde{\mathbf{\Pi}}_1
\\
\widetilde{\mathbf{\Pi}}_2
\\
\vdots
\\
\widetilde{\mathbf{\Pi}}_M
\end{pmatrix}\in\RR^{MdR\times (d^2+1)\widetilde{L}},
\end{align*}
where 
\begin{align*}
&
\widetilde{\mathbf{\Pi}}_m
=
\begin{pmatrix}
\widetilde{\mathbf{\Pi}}_{m11} &  \cdots & \widetilde{\mathbf{\Pi}}_{m1(d^2+1)}
\\
\vdots & \ddots   &\vdots
\\
\widetilde{\mathbf{\Pi}}_{md1} & \cdots & \widetilde{\mathbf{\Pi}}_{md(d^2+1)}
\end{pmatrix}
\quad\text{with}\quad
\left[\widetilde{\mathbf{\Pi}}_{mpq}\right]_{r\ell} 
=
\int_{\hat{D}}\left[\widetilde{\bLambda}_{m}(\bx_r,\by)\right]_{pq}\varphi_q(\by,\hat{\by}_{\ell})d\by,
\\
&
\widetilde{\mathbf{Y}}_m= 
\begin{pmatrix}
\widetilde{\mathbf{Y}}_{m1}
\\
\vdots
\\
\widetilde{\mathbf{Y}}_{md}
\end{pmatrix}
\quad\text{with}\quad
\widetilde{\mathbf{Y}}_{mp}= 
\begin{pmatrix}
[\mathbf{Y}_{m}(\bx_1)]_p
\\
\vdots
\\
[\mathbf{Y}_{m}(\bx_R)]_p
\end{pmatrix},
\quad\text{and}\quad
\widetilde{\mathbf{Z}}_{q}= 
\begin{pmatrix}
[\mathbf{Z}(\hat{\by}_1)]_q
\\
\vdots
\\
[\mathbf{Z}(\hat{\by}_{\widetilde{L}})]_q
\end{pmatrix},
\end{align*}
for all $p\in\{1,\cdots, d\}$, $q\in\{1,\cdots, d^2+1\}$, $m\in\{1,\cdots, M\}$, $\ell\in\{1,\cdots, \widetilde{L}\}$ and $r\in\{1,\cdots, R\}$.
Note that the number of unknowns in the discretized domain is reduced from $(d^2+1)L$ to $(d^2+1)\widetilde{L}$, whereas the sensing matrix $\widetilde{\mathbf{\Pi}}$ is accurate if the estimates of $\hat{\by}_{\ell}$ and $\hat{\bu}_m$  are precise.

\section{Implementation of the  proposed algorithm}\label{s:Implement}

Let us now discuss the implementation details of the solution procedure for the discrete inverse problems \eqref{eq:DisSys} and \eqref{eq:DisSys2}.
Note that the problem \eqref{eq:DisSys} is a joint sparse recovery problem with a prior constraint that the number of non-zero rows  in unknown multi-vector $\mathfrak{X}$ is sparse, whereas problem \eqref{eq:DisSys2} is a classical linear single measurement vector  inverse  problem.
Therefore, the two problems should be addressed separately.
In the following, the compressed sensing approach for joint sparse recovery is  reviewed in Section \ref{ss:CS}, and a modified version of Multiple Sparse Bayesian Learning (M-SBL) algorithm is explained in Section \ref{ss:MSBLImp} as our  joint sparse recovery algorithm to solve \eqref{eq:DisSys}.
In order to solve \eqref{eq:DisSys2}, the detailed implementation of the Constrained Split Augmented Lagrangian Shrinkage Algorithm (C-SALSA) is discussed in Section \ref{ss:Born}.

\subsection{Compressed sensing for joint sparse recovery problems}\label{ss:CS}

Compressed Sensing  (CS) theory is the state of the art in the field of signal processing that enables the recovery of signals beyond the Nyquist limit based on their sparsity \cite{CaRoTa06}. As an example, let us consider the under-determined linear system  $\By = \mathbf{\Pi} \Bx$ that has many solutions. One of the most important innovations of CS is that when the signal $\Bx$ is sparse,  its  accurate recovery is possible using the sparse recovery problem 
\begin{equation}\label{eq:y=Ax0}
\min_{\Bx} \|\Bx\|_0  
\quad\text{subject to }  \By = \mathbf{\Pi} \Bx,
\end{equation}
where $\By \in \mathbb{R}^{J}$, $\mathbf{\Pi} \in \mathbb{R}^{J \times K}$, and $\Bx \in \mathbb{R}^{K}$ with $J<K$ (see, for instance, \cite{CaRoTa06}). Here  $\|\Bx\|_0$ denotes the number of non-zero elements in the vector $\Bx$. 
The uniqueness of the solution to the problem \eqref{eq:y=Ax0} is guaranteed by the condition 
\begin{equation*}%\label{eq:y=Ax0uniq}
\|\Bx\|_0 < \frac{{\rm spark}(\mathbf{\Pi})}{2},
\end{equation*}
where ${\rm spark}(\mathbf{\Pi})$ is the smallest possible number $\ell$ such that there exist $\ell$ linearly dependent columns of $\mathbf{\Pi}$ \cite{DoEl03}. Since \eqref{eq:y=Ax0} is an \emph{NP-hard} problem,  a convex relaxation using $l_1-$minimization widely used in practice is
\begin{equation}\label{eq:y=Ax1}
\min_{\Bx} \|\Bx\|_1 \quad\text{subject to } \By = \mathbf{\Pi} \Bx,
\end{equation}
where $\|\cdot\|_1$ denotes the $l_1-$norm. The beauty of compressed sensing is that \eqref{eq:y=Ax1} provides exactly the same solution as \eqref{eq:y=Ax0} if the so-called \emph{restricted isometry property} (RIP) is satisfied \cite{Candes2008IntroCS}. It has been shown that for many classes of random matrices, the RIP is satisfied with extremely high probability if the number of measurements satisfies  $J\geq c \|\Bx\|_0 \log (K/\|\Bx\|_0)$, where $c$ is an absolute constant \cite{Candes2008IntroCS}.

The Multiple Measurement Vector (MMV) problem is a generalization of the Single Measurement Vector (SMV) problem defined in \eqref{eq:y=Ax0} \cite{chen2006trs,Kim2010CMUSIC}. It is the signal recovery problem to exploit a set of sparse signal vectors that share common non-zero supports, i.e., a set of signal vectors that have \emph{joint sparsity}. Specifically, let $\|\mathfrak{X}\|_0$ denote the number of rows that have non-zero elements in the matrix $\mathfrak{X}$. Then, the MMV problem addresses the following:
\begin{equation} \label{eq:y=Ax0MMV}
\min_{\mathfrak{X}} \|\mathfrak{X}\|_0  
\quad \text{ subject to } \mathfrak{Y} = \mathbf{\Pi}\mathfrak{X},
\end{equation}
where $\mathfrak{Y} \in \mathbb{R}^{J \times M}$, $\mathfrak{X} \in \mathbb{R}^{K \times M}$, and $M$ denotes the number of measurement vectors. 
Since the MMV problem \eqref{eq:y=Ax0MMV} contains more information than the SMV problem \eqref{eq:y=Ax0} (except in the degenerate case when all columns of $\mathfrak{Y}$ are linearly dependent),  it provides better reconstruction results.  Theoretically, problem \eqref{eq:y=Ax0MMV} has unique solution if and only if
\begin{equation}\label{eq:y=Ax0MMVuniq}
\|\mathfrak{X}\|_0 < \frac{{\rm spark}(A)+ \mbox{rank}(\mathfrak{Y})-1}{2},  
\end{equation}
where $\mbox{rank}(\mathfrak{Y})$ denotes the rank of $\mathfrak{Y}$ which may increase with the number of measurement vectors \cite{chen2006trs, Davies2012MMV}. 
 Note that \eqref{eq:y=Ax0MMVuniq} is just an algebraic bound for the noiseless measurements. In practice, the number of measurements can be reduced proportionally to the number of multiple measurement vectors, i.e. 
$J\geq (c /M)\|\mathfrak{X}\|_0 \log (K/\|\mathfrak{X}\|_0 )$ \cite{Kim2010CMUSIC}.  

There are various types of joint sparse recovery algorithms to solve the MMV problem including the convex relaxation \cite{cotter2005ssl, Kim2010CMUSIC, tropp2006ass2,wipf2007ebs}. 
Among those,  the M-SBL  algorithm \cite{wipf2007ebs} is chosen here due to its robustness for noise.  The detailed description of the M-SBL with its specific modification to the problem under investigation is provided in the next section.

\subsection{M-SBL implementation}\label{ss:MSBLImp}

It is worthwhile precising that  M-SBL algorithm was initially derived based on the assumption that the noise  and the unknown signal  follow independent and identically distributed (i.i.d)  zero mean Gaussian distributions. However, recent theoretical analysis substantiates that M-SBL is in fact a sparse recovery algorithm that can be used in deterministic framework without assuming any statistics for the unknown signal (see, for instance, \cite{Wipf2011Latent}). More specifically, it solves the minimization problem
\begin{align}
\label{eq:MSBLcost}
\min_{\mathfrak{X}} \|\mathfrak{Y}-\mathbf{\Pi}\mathfrak{X}\|_F^2 + \zeta \OR_\zeta(\mathfrak{X}),
\end{align}
wherein the penalty function $\OR_\zeta(\mathfrak{X})$ is given by
\begin{align*}
%\label{eq:MSBLcost2}
\OR_\zeta(\mathfrak{X}):=\min_{g_{i}\geq 0, \forall i=1,\cdots,(d^2+1)L} {\rm tr}(\mathfrak{X}^*\bG^{-1}\mathfrak{X}) + M \log \left|\det ( \mathbf{\Pi}\bG \mathbf{\Pi}^* + \zeta \I_{dR})\right|,
\end{align*}
with \emph{tr} indicating the trace of a matrix and the superposed $*$ reflecting a Hermitian conjugate, i.e., $\mathbf{A}^*=\overline{\mathbf{A}}^\top$. Here $\zeta$ is a regularization hyper-parameter controlling the relative weights of the two terms and provides a trade-off between fidelity to the measurements and noise sensitivity, and $\bG\in\RR^{(d^2+1)L\times(d^2+1)L}$ is a diagonal matrix with  entries $[\bG]_{ii}:=g_{i}$  indicating the sparseness of   the respective rows of $\mathfrak{X}$. It is  emphasized that, thanks to the non-separating nature of $\det(\cdot)$, the M-SBL penalty function $\OR_\zeta$ imposes the sparsity more effectively than  the conventional $l^p-$ norms. The interested readers  are referred, e.g., to \cite{Wipf2011Latent} for a detailed topical discussion.

It is interesting to note that by construction the unknown matrix $\mathfrak{X}$ has a special block structure for elasticity imaging unlike the general joint sparse signal recovery problems. In fact, the density $\mathfrak{X}$ is a block matrix composed of ($d^2+1)$ sub-matrices vertically stacked and each one of those has exactly same joint sparsity structure. The sparsity structure of elastostatic problem for $d=2$ is delineated in Figure \ref{fig:JSM}.  
Moreover, since it is really inevitable to avoid measurement noise in practice, it is more appropriate to consider the noisy linear system 
\begin{eqnarray}
\mathfrak{Y}=\mathbf{\Pi}\mathfrak{X}+\mathbf{E},\label{eq:DisSysNoisy}
\end{eqnarray} 
than  the system \eqref{eq:DisSys}.  Here $\mathbf{E}\in\RR^{dR\times M}$ represents additive measurement noise.  
Therefore, the joint sparse recovery problem corresponding to \eqref{eq:DisSysNoisy} subject to the aforementioned structural constraint is given by
\begin{align}
\min_{\mathfrak{X}\in \mathcal{M}_{\rm ad}}
\left\| \mathfrak{X}\right\|_0  
\quad
\text{subject to } 
\left\|
\mathfrak{Y}-\mathbf{\Pi}\mathfrak{X}\right
\|_F 
\leq \epsilon,
\label{eq:JSM}
\end{align}
where $\mathcal{M}_{\rm ad}\subset \RR^{(d^2+1)L\times M}$ is the set of admissible matrices that have aforementioned block structure  and $\epsilon>0$ is a noise dependent  parameter. 
\begin{figure}[!htb] 	
\centering
\includegraphics[width=0.95\textwidth]{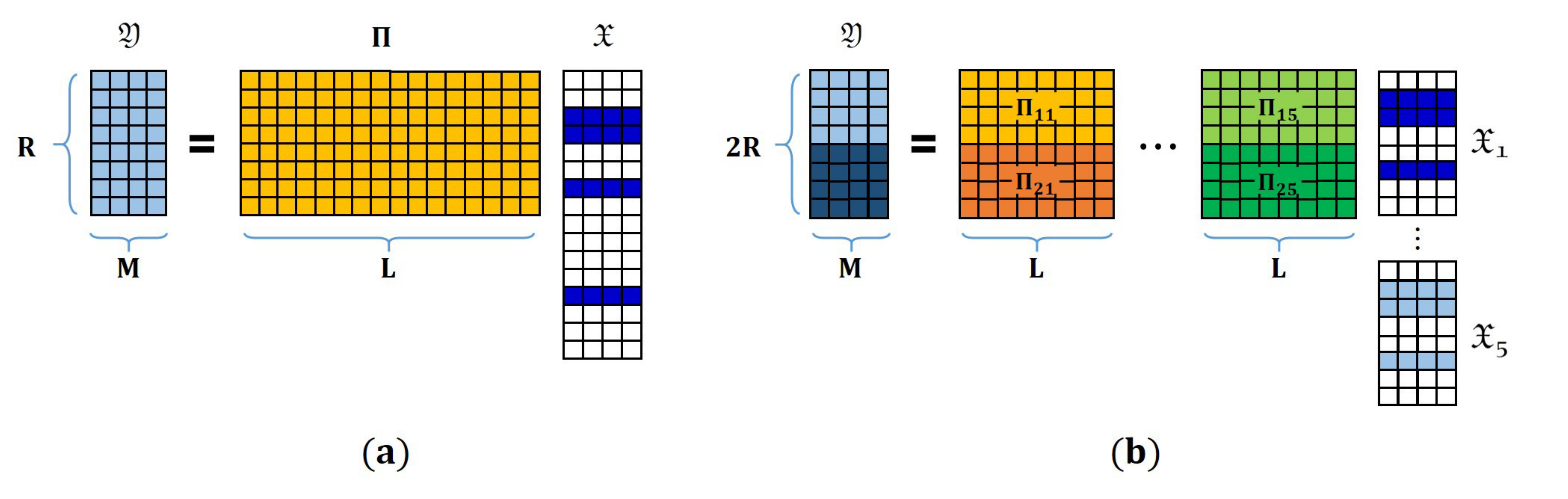}
\caption{Joint sparsity models in $2D$. (a) General problem. (b) Elastostatic problem. 
}
\label{fig:JSM}
\end{figure}

\subsubsection{Signal recovery}

In order to solve the joint sparsity problem \eqref{eq:JSM}, a modified version of M-SBL algorithm with structural constraint is proposed. Precisely,  the general optimization problem \eqref{eq:MSBLcost} is  modified as
\begin{align*}%\label{eq:MSBLcost2b}
\min_{\mathfrak{X}} \|\mathfrak{Y}-\mathbf{\Pi}\mathfrak{X}\|_F^2 + \zeta \OR_\zeta(\mathfrak{X}) + \mathbbm{1}_{\mathcal{M}_{\rm ad}}(\mathfrak{X}),
\end{align*}
where $\mathbbm{1}_{\mathcal{M}_{\rm ad}}$ denotes the indicator function of $\mathcal{M}_{\rm ad}$, i.e.,
\begin{align*}
\mathbbm{1}_{\mathcal{M}_{\rm ad}}(\bw ) :=
\begin{cases} 
0, &  {\bw} \in \mathcal{M}_{\rm ad},
\\ 
\infty, & \text{otherwise.} 
\end{cases}
\end{align*}
 Owing to the structural constraint on unknown density $\mathfrak{X}$, the associated structured sparsities are updated simultaneously using the constraint
\begin{align*}
g_\ell=g_{\ell+L}=\cdots= g_{\ell+ d^2L}, \quad\forall \ell=1,\cdots L.
\end{align*} 
The step-by-step procedure for the modified M-SBL is summarized in Algorithm \ref{alg:PseudocodeMSBL}. 

\begin{algorithm}[!htb]
\caption{Modified M-SBL for sparse signal recovery in elasticity imaging.}
\label{alg:PseudocodeMSBL}
\begin{algorithmic}[1]
\State Set iterations  ${\rm Iter}_{\max}\geq 1$ and threshold $1\gg\varrho>0$.

\State Set $\sigma_{\max}$ to be the largest singular value of $\mathbf{\Pi}$.  

\State  Set $k:=0$ and $\zeta^{(0)}:=10 \times \sigma_{\max}^2$.

\State Set $g^{(0)}_{\ell}:=1$ for $\ell=1,2,\cdots,(d^2+1)L$ and $\bG^{(0)}:={\rm diag}\left(g_{1}^{(0)},g_{2}^{(0)},\cdots, g^{(0)}_{(d^2+1)L}\right)$.

\State Set ${\bpi}_\ell:= \Big([\mathbf{\Pi}]_{1\ell}, \cdots, [\mathbf{\Pi}]_{(dR)\ell}\Big)^\top$ and ${\bpi}^*_\ell:= \left([\mathbf{\Pi}]^*_{1\ell}, \cdots, [\mathbf{\Pi}]^*_{(dR)\ell}\right)^\top$.

\For{$k=1,\dots, {\rm Iter}_{\max}$}
 
    \State Set $\mathbf{F}^{(k-1)}:=\left(\mathbf{\Pi}\bG^{(k-1)}\mathbf{\Pi}^* + \zeta^{(k-1)}\I_{dR} \right)^{-1}$.
   
    \State Update $\mathfrak{X}^{(k)}=\bG^{(k-1)}\mathbf{\Pi}^*\mathbf{F}^{(k-1)} \mathfrak{Y}$.
   
    \State Update: for $\ell=1,\cdots,L,$
    $$
    \ds g_{\ell}^{(k)}=g_{\ell+L}^{(k)}=\cdots=g_{\ell+d^2L}^{(k)} =\sqrt{{\ds\sum_{q=1}^{d^2+1}\sum_{m=1}^M\left| \left[\mathfrak{X}_q^{(k)}\right]_{\ell m}\right|^2}\Big/M{\ds\sum_{p=0}^ {d^2L}\bpi_{\ell+p}^*\mathbf{F}^{(k-1)} \bpi_{\ell+p} }}.
    $$
    
    \State Set $\bG^{(k)}:={\rm diag}\left(g_{1}^{(k)},g_{2}^{(k)},\cdots, g^{(k)}_{(d^2+1)L}\right)$.
    
    \If{$\ds {g_\ell^{(k)}}/{\max_{\ell}\left(g_\ell^{(k)}\right)} < \varrho$}  
     
        \State $g_\ell^{(k)}=0.$
    
    \EndIf
    
    \State Update $\ds\zeta^{(k)} = \sqrt{{\|\mathfrak{Y}-\mathbf{\Pi}\mathfrak{X}^{(k)}\|_F^2}\Big/{M \,{\rm tr}\left(\mathbf{F} ^{(k-1)}\right)}}$.
    
\EndFor
 \Return $\hat{\mathfrak{X}}:=\mathfrak{X}^{(k)}$.
\end{algorithmic}
\end{algorithm}

\subsubsection{Preconditioning}

Recall that the problem \eqref{eq:DisSysNoisy} is severely ill-posed if the inclusions are extended and not really sparse inside the background domain due to the intrinsic ill-posedness of the elasticity imaging problem. Moreover, the sensing matrix, which is associated to a physical system, has a coherence structure and its columns are not completely incoherent. This affects the performance of the sparsity based recovery algorithms. Therefore, it is desirable to introduce a surgical preconditioning procedure before executing M-SBL algorithm. For this, the singular value decomposition of the sensing matrix is considered as $\mathbf{\Pi} = \mathbf{V}\mathbf{\Sigma}\mathbf{W}^*$. Here $\mathbf{\Sigma}\in\RR^{dR\times (d^2+1)L}$ is such that $[\mathbf{\Sigma}]_{rr}=:\sigma_{r}$, for all $r=1,\cdots, dR$, are the singular values  of $\mathbf{\Pi}$ and $[\mathbf{\Sigma}]_{r\ell}=0$ for all $r\neq \ell$. The matrices $\mathbf{V}\in\RR^{dR\times dR}$ and  $\mathbf{W}\in\RR^{(d^2+1)L\times (d^2+1)L}$ are unitary and their columns are respectively the left and right singular vectors of $\mathbf{\Pi}$. Consequently, a preconditioning weight matrix $\mathbf{P}\in\RR^{dR\times dR}$ can be introduced as
\begin{align*}
\mathbf{P}= (\mathbf{\Sigma}^2+\theta\I_{dR})^{-1/2}\mathbf{V}^*,
\end{align*} 
with $\theta$ being a thresholding parameter  (refer, for instance, to \cite{Jin2012precond, Lee2011CDOT}).  The M-SBL algorithm can then be applied to the regularized problem 
\begin{align}
\label{eq:preconProblem}
\mathbf{P}\mathfrak{Y}=\mathbf{P}\mathbf{\Pi}\mathfrak{X}+\mathbf{P}\mathbf{E}.
\end{align}

\subsubsection{Support identification}

The application of M-SBL Algorithm \ref{alg:PseudocodeMSBL}  renders the  unique minimizer  $\hat{\mathfrak{X}}$ to  the constraint optimization problem   \eqref{eq:JSM}.  Having recovered the sparse signal vector, one can  identify the support set $\cup_{n=1}^ N D_n$ by collecting all $\by_\ell$ such that $[\widetilde{\mathbf{X}}_m(\by_\ell)]_q$ is nonzero for all $q=1,\cdots, d^2+1$ and $m=1,\cdots,M$. In other words, it suffices to look for $\ell\in \{1,\cdots, L\}$ such that $[\hat{\mathfrak{X}}_q]_{\ell m}$ is non-zero for all $q=1,\cdots, d^2+1$ and $m=1,\cdots,M$.  Towards this end, set
\begin{align}
\label{eq:SpecMSBL2}
\hat{D} := \left\{\by_{\ell} \;\Big\vert\; {\psi_{\ell}}\Big/{\ds\max_{\ell\in\{ 1,\cdots, L\} }(\psi_\ell)} > \xi \right\},
\end{align}
where $\xi$ is a pruning parameter and $\psi_\ell$ is  defined by
\begin{align*}
%\label{eq:SpecMSBL}
\psi_{\ell} := \sqrt{\sum_{q=1}^{d^2+1} \sum_{m=1}^{M} \left|\left[\hat{\mathfrak{X}}_q\right]_{\ell m}\right|^2}, \quad 1\leq  \ell \leq L.
\end{align*}
Note that $\psi_{\ell} $ indicates the sparseness of the $l$-th row.  Due to the numerical implementation,  the values of $[\hat{\mathfrak{X}}_q]_{\ell m}$ cannot reach  zero absolutely, though they may be very small. Consequently,  this pruning step is indispensable to sweep away the values smaller than a predefined threshold depending on the noise level and numerical discretization.

\subsection{Parameter reconstruction}\label{ss:Born}

For the quantitative evaluation of the Lam\'{e} parameters of $\hat{D}$, one needs to solve the discrete system \eqref{eq:DisSys2}. This can be done by formulating the constrained optimization problem 
\begin{align}
\label{eq:CondRecon2}
{\rm arg}\min \limits_{\mathfrak{Z}} \, \widetilde{\zeta}\|\mathfrak{Z}\|_{1}
\,\text{ subject to }\,  \| \widetilde{\mathbf{Y}}-\widetilde{\mathbf{\Pi}}\mathfrak{Z}\|_2 \leq \eta, 
 \,\,
 \theta_{\min} \leq  [\mathfrak{Z}]_i\leq  \theta_{\max},
 \,\,
  i=1,\cdots, (d^2+1)\widetilde{L},
\end{align}
where $l_1-$ penalty is enforced in order to achieve noise robust reconstruction and  the constraint 
\begin{align}
\theta_{\min}\leq [\mathfrak{Z}]_i \leq \theta_{\max},\label{eq:BoxCons}
\end{align}
emerges from the assumption H3. Here $\theta_{\min}$ and $\theta_{\max}$ are real numbers such that $\theta_{\min}\leq\theta_{\max}$ and $\widetilde{\zeta}$ is the constraint weight. There are several algorithms available in the literature that are tailored to solve such constrainted optimization problems and any one of them can be deployed to resolve \eqref{eq:CondRecon2}. In this article, the C-SALSA by Afonso, Bioucas-Dias, and Figueiredo \cite{Afonso2011CSALSA} is implemented.  A pseudo-code implementation  of C-SALSA  for \eqref{eq:CondRecon2} is furnished in Algorithm \ref{alg:PseudocodeCSALSA}.  Beforehand, the sensing matrix $\widetilde{\mathbf{\Pi}}$ is normalized so that each one of its columns has a unit $l_2-$norm, however,  the normalized matrix is still denoted by $\widetilde{\mathbf{\Pi}}$  by abuse of notation. 

The idea of C-SALSA  is to transform the constrained optimization problem into an unconstrained problem first. Then, the resulting problems is further transformed using a variable splitting operation before finally being resolved using \emph{Alternating Direction Method of Multipliers} (ADMM). The interested readers are refered to \cite{Afonso2011CSALSA} for a topical review and detailed description of C-SALSA.

Let $\mathcal{B}_\eta(\widetilde{\mathbf{Y}})$ be the Euclidean ball in $\RR^{MdR}$ centered at $\widetilde{\mathbf{Y}}$ and radius $\eta$.  Then, the problem \eqref{eq:CondRecon2} can be seen as the unconstrained problem  (see \cite{Afonso2011CSALSA})
\begin{align}
\label{eq:UnconstrainedProb}
\min_{\mathfrak{Z}} \widetilde{\zeta}\| \mathfrak{Z}\|_{1}+\mathbbm{1}_{\mathcal{B}_\eta(\widetilde{\mathbf{Y}})} \left(\widetilde{\mathbf{\Pi}}\mathfrak{Z}\right),
\end{align}
where $\mathbbm{1}_{\mathcal{B}_\eta(\widetilde{\mathbf{Y}})}$ is the indicator function of ${\mathcal{B}_\eta(\widetilde{\mathbf{Y}})}$, i.e., 
\begin{align*}
\mathbbm{1}_{\mathcal{B}_\eta(\widetilde{\mathbf{Y}})}\left(\bw\right)
:=
\begin{cases}
0, &  \bw\in {\mathcal{B}_\eta(\widetilde{\mathbf{Y}})},
\\
\infty, &\text{otherwise}.
\end{cases}
\end{align*}

Let us introduce the mappings $f_1:\RR^{(d^2+1)\widetilde{L}}\to \RR$ and $f_2:\RR^{Md\widetilde{L}}\to \RR$ by 
$ f_1(\bs):=\|\bs\|_1$  and  $f_2(\bs):=\mathbbm{1}_{\mathcal{B}_\eta(\widetilde{\mathbf{Y}})}(\bs)$,
and the corresponding \emph{Moreau proximal mappings} $\Psi_{\tau f_1}: \RR^{(d^2+1)\widetilde{L}} \to  \RR^{(d^2+1)\widetilde{L}}$  and $\Psi_{\tau f_2}: \RR^{Md\widetilde{L}} \to  \RR^{Md\widetilde{L}}$ by 
\begin{align*}
\Psi_{\tau f_1} (\bs)&:= {\rm arg}\min_{\bv} \frac{1}{2}\|\bv-\bs\|_2^2+\tau\|\bv\|_1,
\\
\Psi_{\tau f_2} (\bs)&:= {\rm arg}\min_{\bw} \tau \mathbbm{1}_{\mathcal{B}_\eta(\widetilde{\mathbf{Y}})}(\bw)+\frac{1}{2}\|\bw-\bs\|_2^2.
\end{align*}
Refer, for instance, to \cite{Combettes} and articles cited therein for  details on Moreau proximal mappings. It is  worthwhile mentioning that   $\Psi_{\tau f_1}$ with $f_1$ being the $l_1-$regularizer turns out to be simply a soft thresholding \cite{Afonso2011CSALSA}, i.e., 
\begin{align*}
\Psi_{\tau f_1}(\bs)={\rm soft}(\bs,\tau),
\end{align*}
where  ${\rm soft}(\bs,\tau)$ reflects the  component-wise operation  
\begin{align*}
[\bs]_j = {\rm sign}\big([\bs]_j\big)\max\big\{\big|[\bs]_j\big|-\tau,\;0\big\}.
\end{align*}
In numerical implementation,  the regularization parameter $\tau$ is chosen to be  $\tau=0.1|\bar{\mathfrak{Z}}_0|$, where $|\bar{\mathfrak{Z}}_0|$ is the average value of $|\mathfrak{Z}|$ at the zeroth iteration in Algorithm \ref{alg:PseudocodeCSALSA}. The threshold $\eta$ is fixed at $\eta=0.3\|\widetilde{\mathbf{Y}}\|_2$.   The data fidelity parameter $\widetilde{\zeta}$ is manually selected as the optimal choice from the set $ \{ 8,4,2,1,1/2,1/4,1/8\}$.  The parameters $\theta_{\min}$ and $\theta_{\max}$ in \eqref{eq:BoxCons} are set to $-\infty$ and $+\infty$, respectively. Even in this (unconstrained) setup, the constraint in \eqref{eq:BoxCons} is necessary to introduce additional variable splitting, which allows much faster convergence. This type of additional splitting is quite often used in ADMM.  Finally, the relative change of the cost function in \eqref{eq:CondRecon2} is used as a stopping criterion, i.e., the algorithm is executed until $\left| \left(C_{k}-C_{k-1}\right)/C_{k} \right| < 10^{-4}$ is satisfied, where $C_{k}$ is the cost function at the $k$-th iteration. With these choices of parameters, the relevant pseudo-code implementation of C-SALSA for the resolution of the  unconstrained problem \eqref{eq:UnconstrainedProb} is provided in Algorithm \ref{alg:PseudocodeCSALSA}. 
 
\begin{algorithm}[!htb]
\caption{C-SALSA for parameter recovery in elasticity imaging.}
\label{alg:PseudocodeCSALSA}
\begin{algorithmic}[1]
\State Set $k=0$, choose $\tau>0$.

\State Set $\Ba_0^{(i)}=\Bb_0^{(i)}=0$, for $i=1,2$.

\Repeat
\State $\br_k = \widetilde{\zeta}(\Ba_k^{(1)} + \Bb_k^{(1)}) + \widetilde{\mathbf{\Pi}}^*(\Ba_k^{(2)} + \Bb_k^{(2)})$.

\State $\mathfrak{Z}_{k+1} = \left[(1+{\widetilde{\zeta}}^2) \I+\widetilde{\mathbf{\Pi}}^*\widetilde{\mathbf{\Pi}}\right]^{-1}\br_k$.

\State $\Ba_{k+1}^{(1)} = \Psi_{\tau f_1} ( \widetilde{\zeta} \mathfrak{Z}_{k+1} - \Bb_k^{(1)} )$.

\State $\Ba_{k+1}^{(2)} = \Psi_{f_2} ( \widetilde{\mathbf{\Pi}} \mathfrak{Z}_{k+1} - \Bb_k^{(2)} )$.

\State $\Bb_{k+1}^{(1)} = \Bb_{k}^{(1)} - \widetilde{\zeta} \mathfrak{Z}_{k+1} + \Ba_{k+1}^{(i)}$.

\State $\Bb_{k+1}^{(2)} = \Bb_{k}^{(2)} - \widetilde{\mathbf{\Pi}}\mathfrak{Z}_{k+1} + \Ba_{k+1}^{(i)}$.

\State $k \longleftarrow k + 1$.
\Until{Stopping criterion is satisfied.}
\end{algorithmic}
\end{algorithm}

\section{Numerical validation of reconstruction scheme}\label{s:NS}

In this section, some numerical experiments are performed in order  to validate the proposed reconstruction scheme.  Let us first provide the details of the numerical scheme for the  forward model in Section \ref{ss:forward}. The examples of the reconstruction of different inclusions are furnished in Section \ref{ss:tests}.

\subsection{Forward solver}\label{ss:forward}

In order to solve the forward problem for data acquisition, the boundary layer potential technique is used together with the so-called \emph{Nystr\"{o}m discretization scheme}. Let us briefly fix the ideas about the resolution of the forward problem. For simplicity, a two- dimensional case is considered. Note that the vector space  $\Psi$ in two-dimensions is given by 
$$
\Psi={\rm Span}
\left\{
\begin{pmatrix}
1\\0
\end{pmatrix},
\begin{pmatrix}
0\\1
\end{pmatrix},
\begin{pmatrix}
x_2\\-x_1
\end{pmatrix}
\right\}.
$$
For $M$ distinct points $\bz_1,\cdots, \bz_M\in\RR^2\setminus \overline{\Omega}$, generate 
\begin{equation}\label{eqn:truesolution}
\bU_m(\bx)=
\bGam(\bx-\bz_m)\cdot
\begin{pmatrix}
    1\\0
    \end{pmatrix}
    +\alpha_1
    \begin{pmatrix}
    1\\0
    \end{pmatrix}
    +\alpha_2
    \begin{pmatrix}
    0\\1
    \end{pmatrix}
    +\alpha_3
    \begin{pmatrix}
    x_2\\-x_1
    \end{pmatrix},
\end{equation}
where $\alpha_1,\alpha_2,\alpha_3\in\RR$ are chosen in such a way that $\bU_m |_{\partial\Omega}\in L^2_{\Psi}(\partial\Omega)$, i.e.,
\begin{equation*}
\int_{\partial\Omega}\bU_m\cdot
\begin{pmatrix}
1\\0
\end{pmatrix}d\sigma=
\int_{\partial\Omega}\bU_m\cdot
\begin{pmatrix}
0\\1
\end{pmatrix}d\sigma=
\int_{\partial\Omega}\bU_m\cdot
\begin{pmatrix}
x_2\\-x_1
\end{pmatrix}d\sigma=0.
\end{equation*}
The surface traction $\bg_m$, $m=1,\cdots, M$, is then calculated by the relation $\bg_m={\partial\bU_m}/{\partial\bnu}|_{\partial\Omega}$. Remark that $\OL_{\lambda_0,\mu_0}[\bU_m]=0$ in $\Omega$ and consequently $\bg_m\in L^2_{\Psi}(\partial \Omega)$.

In order to generate the displacement field $\bu_m$ in the presence of inclusions, $D_1,\cdots, D_N$, the transmission problem \eqref{Prob_Trans2} is solved using a layer potential technique. Towards this end,  the single layer potential  associated to the linear elasticity operator $\OL_{\lambda_0,\mu_0}$  is defined by
\begin{eqnarray*}
\mathcal{S}_\Omega[\bphi](\bx):=\int_{\partial \Omega}\bGam(\bx-\by)\cdot\bphi(\by)d\sigma(\by), \qquad \bx\in\RR^2,  \quad \bphi\in L^2(\partial \Omega)^2.
\end{eqnarray*}
It is well known (see, for instance,  \cite{Dahlberg}) that $\mathcal{S}_\Omega[\bphi]$ satisfies the  jump relations  
\begin{align*}
\frac{\partial}{\partial \bnu}\mathcal{S}_\Omega[\bphi]\Big|_{\pm}(\bx)=\left(\pm\frac{1}{2}\mathcal{I}+\mathcal{K}^*_\Omega\right)\bphi(\bx), \qquad{\rm a.e.}\quad\bx\in\partial \Omega,
\end{align*}
where $\Kcal^*_\Omega$ is the $L^2-$adjoint operator of $\Kcal_\Omega$ and is defined by
\begin{equation*}
\Kcal^*_\Omega[\bphi](\bx) := {\rm p.v.} \ds\int_{\partial \Omega}\frac{\partial}{\partial\bnu_\bx}\bGam (\bx-\by)\cdot\bphi(\by)d\sigma(\by),\quad{\rm a.e.}\qquad\,\bx\in\partial\Omega, \quad\bphi\in L^2(\partial\Omega)^2.
\end{equation*}
Then, the solution to the system\eqref{Prob_Trans2} can be represented as (see, e.g., \cite[Theorem 6.15]{AK-Book})
\begin{equation}\label{eqn:representationtransmission}
\bu_m(\bx)=\begin{cases}
         \mathcal{S}_\Omega[\bfeta_m]+\mathcal{S}_D[ \bpsi_m](\bx) , & \mbox{if } \bx\in\Omega\setminus\overline{D}, 
         \\
         \widetilde{\mathcal{S}}_D[\bphi_m](\bx), & \bx\in D,
       \end{cases}
\end{equation}
where  $(\bphi_m,\bpsi_m,\bfeta_m)\in L^2(\partial D)^2\times L^2_{\Psi}(\partial\Omega)^2\times L^2(\partial D)^2$ is the unique solution to  
\begin{equation}\label{eqn:systemtwo}
\begin{bmatrix}
 \ds\ds \widetilde{\mathcal{S}}_D|_{\partial\D} 
 & 
\ds -\mathcal{S}_D|_{\partial\D} 
 & 
 \ds -\mathcal{S}_\Omega|_{\partial\D} 
  \\
 \ds \left(-\frac{1}{2}\mathcal{I}+\widetilde{\Kcal}^*_D\right) 
 & 
 -\ds\left(\frac{1}{2}\mathcal{I}+\Kcal^*_D\right) 
 & 
 -\ds \frac{\partial}{\partial\bnu}\mathcal{S}_\Omega|_{\partial D}
  \\
 \ds 
  0
  & 
  \ds\frac{\partial}{\partial\bnu}\mathcal{S}_D|_{\partial\Omega} 
  & 
  \ds \left(-\frac{1}{2}\mathcal{I}+\mathcal{K}_{\Omega}^*\right)
\end{bmatrix}
\begin{bmatrix}
  \bphi_m
  \\
  \bpsi_m
  \\
  \bfeta_m
\end{bmatrix}
=\begin{bmatrix}
   0
   \\
   0
   \\
   \bg_m
 \end{bmatrix},
\end{equation}
subject to the constraint $\bu_m\in L^2_\Psi(\partial\Omega)$, thanks to the transmission and the boundary conditions. Here $\widetilde{\mathcal{S}}_D$ and $\widetilde{\mathcal{K}}_D^*$ are the operators related to the interior parameters $(\lambda, \mu)$.

The aim here is to solve the system \eqref{eqn:systemtwo} for $(\bphi_m,\bpsi_m,\bfeta_m)$  and then evaluate  $\bu_m|_{\partial \Omega}$ using representation \eqref{eqn:representationtransmission}. 
In order to numerically  solve system \eqref{eqn:systemtwo}, express $\mathcal{K}^*_\Omega$ as
\begin{equation}\label{eqn:parametrizedKomeg}
\Kcal^*_{\Omega}[\f](\bx)={\rm p.v.}\int_{0}^{1}\frac{\partial}{\partial\bnu_{\bx }}\bGam(\bx -\bx(t))\cdot\f(\bx(t)) |\bx'(t)| dt, \quad\, \bx\in\partial D,
\end{equation}
where $\bx(t)$ is a parametrization of $\partial \Omega$.
Then the Nystr\"{o}m discretization, with $P$ boundary points $\{\bx_p\}_{p=1}^{P}$ and weights $w_p$,  renders 
\begin{equation}\label{eqn:discretizationKOmeg}
\Kcal^*_{\Omega}[\f](\bx)\approx\sum_{p=1}^{P}\frac{\partial}{\partial\bnu_\bx}\bGam(\bx-\bx_p)\cdot\f_p |\bx'_p| w_p, \quad\,\bx\in\partial\Omega,
\end{equation}
where  $\f_p=\f(\bx_p)$ for $p=1,\cdots, P$.  
Here, the simplest quadrature rule 
$$
\ds 
\int_{0}^{1}\f(t)dt\approx\sum_{p=1}^{P}\frac{1}{P}\f\left(\frac{p}{P}\right),
$$ 
is used. It is emphasized that the integral \eqref{eqn:parametrizedKomeg}  is defined as the Cauchy principal value. Thus, the singularity of ${\partial}[\bGam(\bx-\bx_p)]/{\partial\bnu_\bx}$ at $\bx=\bx_p$ should be evaluated in the sense of the Cauchy principle value. The numerical computation of $\Kcal^*_\Omega$ can be realized using \eqref{eqn:discretizationKOmeg}. Similarly,  $\Kcal_{\Omega}$, $\mathcal{S}_\Omega$ and $\mathcal{D}_{\Omega}$ can be descretized as
\begin{align*}
\Kcal_{\Omega}[\f](\bx) 
&
\approx\sum_{p=1}^{P}\frac{\partial}{\partial\bnu(\bx_p)}\bGam(\bx-\bx_p)\cdot\f_p|\bx'_p| w_p, \quad\,\bx\in\partial\Omega,
\\
\mathcal{S}_{\Omega}[\f](\bx)
&
\approx\sum_{p=1}^{P}\bGam(\bx-\bx_p)\cdot\f_p |\bx'_p| w_p, \quad\bx\in \Omega,
\\
\D_{\Omega}[\f](\bx)
&
\approx\sum_{p=1}^{P}\frac{\partial}{\partial\bnu(\bx_p)}\bGam(\bx-\bx_p)\cdot\f_p |\bx'_p| w_p, \quad\bx\in \Omega.
\end{align*}
Consequently, the integral system \eqref{eqn:systemtwo} can be discretized and solved for $(\bphi_m,\bpsi_m,\bfeta_m)$. 
If only sparsely sampled data $(\bu_m-\bU_m)$ are available, 
the preprocessing of the data $(\bu_m-\bU_m)$ using the Caler\'{o}n preconditioner $(-\mathcal{I}/2+\mathcal{K}_\Omega)$ can be done in the similar fashion after being interpolated to  dense samples
as discussed in Section \ref{ss:Calderon}.

\subsection{Numerical experiments}\label{ss:tests}
 
For numerical examples, let the background domain $\Omega$ to be an ellipse of semi-major and semi-minor axes $10mm$ and $7mm$ respectively with shear and compression moduli  $\mu_0=\lambda_0=1GPa$. Three different kinds of inclusions are considered for numerical experiments. Precisely, the examples of sparse, extended and thin or worm-like inclusions are taken into account.  The sparse inclusions are modeled with three unit disks.  The extended inclusion is modeled with a non-convex kite shaped domain with the size comparable to that of the background domain in order of magnitude. By thin or worm-like inclusions, it is meant that  one dimension of the inclusions is  much smaller than the other dimension. The examples of straight and curved thin inclusions are dealt with. The test geometries are delineated in Figure \ref{fig:targets}.
The field of view is discretized to have a grid size $1/3mm$ for all reconstructions. The Lam\'{e} parameters of the three inclusions in the sparse case are both fixed at $7GPa$ for the leftmost inclusion, $2GPa$ for the inclusion in the middle, and $2.5GPa$ for the rightmost inclusion. For the rest of the examples, the Lam\'{e} parameters of the targets are both fixed at $2GPa$.  
\begin{figure}[!htb] 	
\centering
\includegraphics[width=0.95\textwidth]{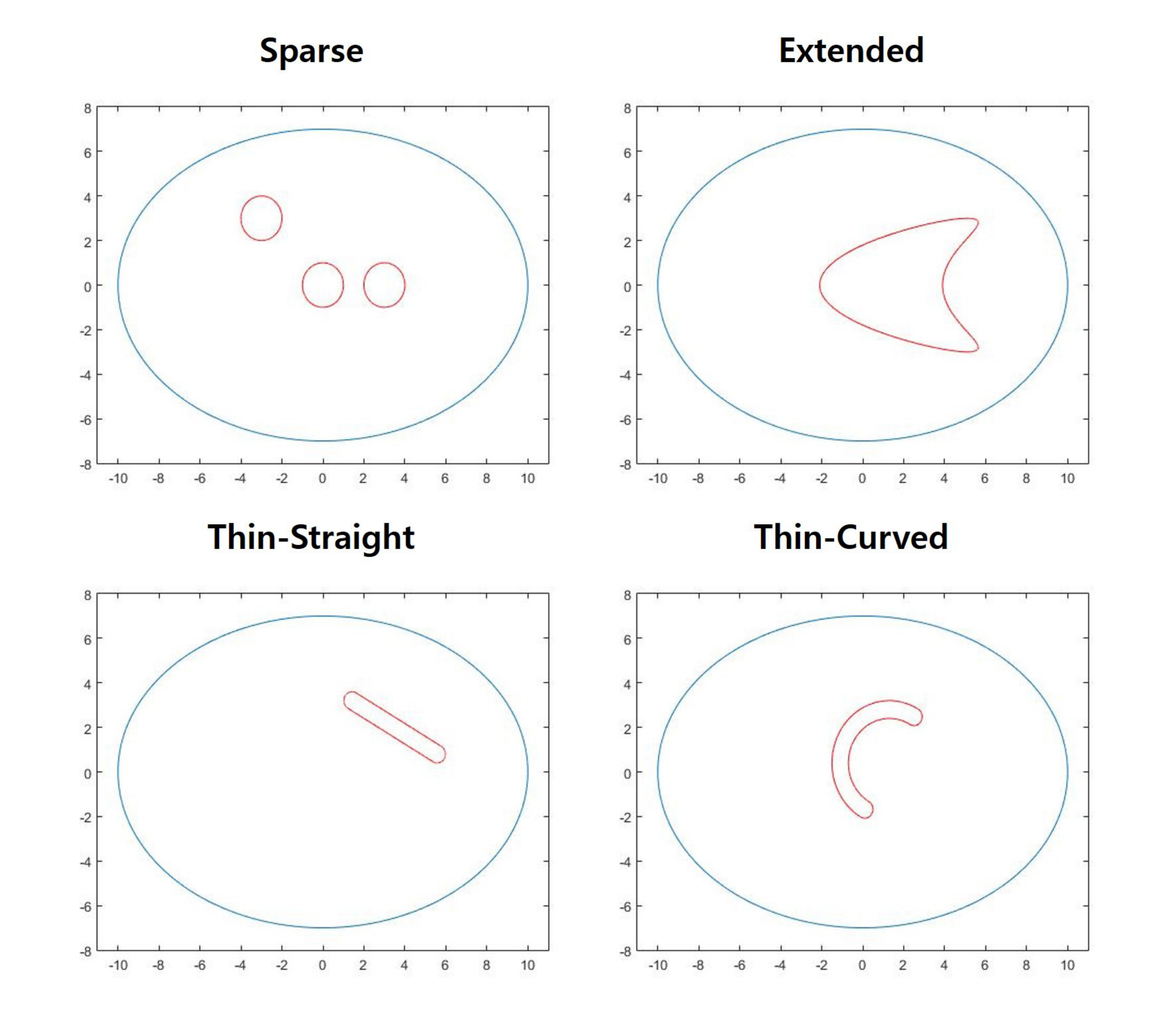} 
\caption{Geometric configuration and different test inclusions. }
\label{fig:targets}
\end{figure}

\subsubsection{Parameter selection}

For all experiments, four measurement sets are used, i.e., $M=4$. Accordingly, points $\bz_1=(12,11)$, $\bz_2=(9,-11)$, $\bz_3=(-1,8)$ and $\bz_4=(-50,0)$ are used in \eqref{eqn:truesolution} to define  $\{\bU_m\}_{m=1}^4$ and $\{\bg_m\}_{m=1}^4$. The forward data is acquired using the numerical  scheme described in Section \ref{ss:forward}. $P=2000$ discretization points on $\partial\Omega$ and $\partial D_n$, for each $n$,  are used for the example of sparse inclusions, and $P=5000$ points are used for rest of the examples. Three different sets of the uniformly distributed full view measurement points $\{\bx_r\}_{r=1}^R$ with $R=100$, $32$, $16$ and  a set of limited view (with angle $3\pi/2$) measurement  points  on $\partial \Omega$ with $R=16$  are taken into account. The latter case is indicated hereinafter by $R=16p$. The measurement setups are depicted in Figure \ref{fig:meas}. An additive Gaussian noise with signal-to-noise ratio $40dB$ was added to the boundary measurement vectors $(\bu_m-\bU_m)\big|_{\partial\Omega}$ for all simulations.  
\begin{figure}[!htb] 	
\centering
\includegraphics[width=0.95\textwidth]{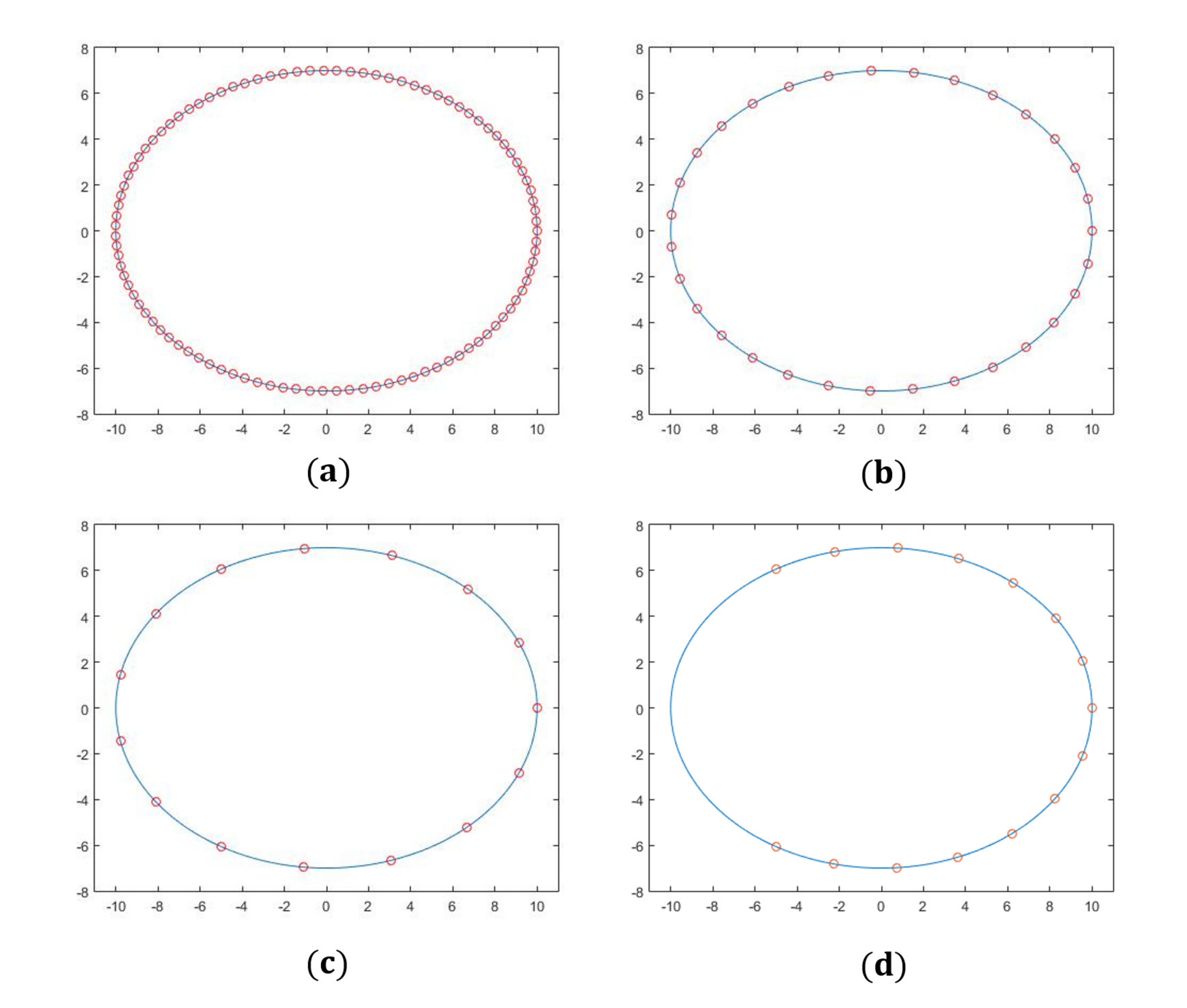} 
\caption{Configuration of the measurement points on the boundary of $\Omega$. (a) $R=100$. (b) $R=32$. (c) $R=16$.  (d) $R=16p$.}
\label{fig:meas}
\end{figure}

In order to recover the density $\mathfrak{X}$ over the support $\hat{D}$, the modified M-SBL Algorithm \ref{alg:PseudocodeMSBL} is applied on the preconditioned problem \eqref{eq:preconProblem} using  regularization parameter $\theta=10^{-2}\times\sigma_{\max}^2$, where $\sigma_{\max}$ denotes the maximum singular value of the sensing matrix $\Pi$. The threshold parameter $\varrho$ in Algorithm \ref{alg:PseudocodeMSBL} is set to be $\varrho=10^{-3}$. 

For support identification, the pruning parameter $\xi$ in \eqref{eq:SpecMSBL2} is set to be $\xi=0$, i.e., small values are not pruned out and the obtained information is fully utilized in order to avoid a sub-optimal selection of $\xi$. The box constraint parameters $\theta_{\min}$ and $\theta_{\max}$ in \eqref{eq:BoxCons} are set to be very large so that $[\mathfrak{Z}]_i$ can simply assume values in $(-\infty,+\infty)$. 

The selected optimal values of the parameters for Algorithms \ref{alg:PseudocodeMSBL} and  \ref{alg:PseudocodeCSALSA} are summarized in Table \;\ref{tbl:ParameterSimul}. These parameters are used for all examples except for data fidelity parameter  $\widetilde{\zeta}$, which is gradually decreased in value with respect to the decrease in the number of measurement  points. 
\begin{table}[!htb]
\footnotesize
\caption{Parameters choices for numerical simulations.}
\label{tbl:ParameterSimul}
\vspace*{0.2cm}
\centerline{
\renewcommand{\arraystretch}{1.3}
\begin{tabular}{|c|c|c|c|c|} \hline
                Proposed Method & Sparse target & Thin-Straight target & Thin-Curved target  &  Extended target \\
\hline \hline
\multirow{2}{*}{Step 1: M-SBL}& $\mbox{Iter}_{\max}=50$   & $\mbox{Iter}_{\max}=50$   & $\mbox{Iter}_{\max}=50$  & $\mbox{Iter}_{\max}=50$ \\
                & with preconditioning   & with preconditioning   & with preconditioning & with preconditioning\\
\cline{1-5}
\multirow{3}{*}{Step 2: C-SALSA} & $\tau= 0.1|\bar{\mathfrak{Z}}_0|$  & $\tau= 0.1|\bar{\mathfrak{Z}}_0|$  & $\tau=0.1|\bar{\mathfrak{Z}}_0|$   & $\tau=0.1|\bar{\mathfrak{Z}}_0|$\\
                & $\eta=0.3\|\widetilde{\mathbf{Y}}\|_2$ & $\eta=0.3\|\widetilde{\mathbf{Y}}\|_2$ & $\eta=0.3\|\widetilde{\mathbf{Y}}\|_2$ & $\eta=0.3\|\widetilde{\mathbf{Y}}\|_2$ \\
                & $\widetilde{\zeta} = 2~(R=100)$ & $\widetilde{\zeta} = 4~(R=100)$ & $\widetilde{\zeta} = 4~(R=100)$  & $\widetilde{\zeta}= 8~(R=100)$\\
                 & $\widetilde{\zeta} = 1/2~(R=32)$ & $\widetilde{\zeta} = 2~(R=32)$ & $\widetilde{\zeta} = 1~(R=32)$  & $\widetilde{\zeta}= 4~(R=32)$\\
                & $\widetilde{\zeta} = 1/4~(R=16)$ & $\widetilde{\zeta} = 1~(R=16)$ & $\widetilde{\zeta} = 1/2~(R=16)$  & $\widetilde{\zeta}= 2~(R=16)$\\	
                & $\widetilde{\zeta} = 1/4~(R=16p)$ & $\widetilde{\zeta} = 1/2~(R=16p)$ & $\widetilde{\zeta} = 1/2~(R=16p)$  & $\widetilde{\zeta}= 2~(R=16p)$\\	
\hline
\noalign{\hrule height 0.5pt}
\end{tabular}
}
\normalsize
\end{table}

\subsubsection{Simulation results}

The reconstructed shear and compression moduli for different inclusions together with estimated support of the inclusions are provided in Figures \ref{fig:recon-3}--\ref{fig:recon-kite} for sparse, thin straight, thin curved and extended inclusions respectively. For all the listed inclusions, the results  are furnished with different configurations of measurement points (with $R=100$, $32$, $16$, $16p$) as precised earlier. When $R=100$, the proposed algorithm clearly recovered the structures of the inclusions and their  constitutive parameters in all cases. For instance, for sparse inclusions, even though the parameter values have been varied by tuning the optimization parameters, their relative relationships remained the same so that the leftmost inclusion always appears to have the highest value and the middle one has the lowest value (see Figure \ref{fig:recon-3}). 
It is observed that the overall reconstruction performance gradually suffers when the number of measurement points decreases.  Nevertheless, the results corresponding to $R=32$ are comparable to those of $R=100$. Even when $R=16$, the simulations are mostly able to indicate the crude shapes of  the inclusions.

 However, when the measurement points only cover the partial aperture ($R=16p$ and $3\pi/2$ angle of view), the results are distorted. 
 The reconstructions for thin and extended targets show comparatively less accurate
results than those for the sparse targets.
%The reconstructions for thin and extended targets show similar results as those for the sparse targets. 
  It is worthwhile to mention that M-SBL was still able to localize the anomalies even in the deteriorated conditions. In the deteriorated cases from the partial aperture, although the M-SBL algorithm recovered the locations outside the expected regions, the estimated M-SBL values are relatively higher inside and near the boundary of the inclusions than spurious detected regions outside the inclusions (see Figure \ref{fig:MSBL-supp}). The points outside the inclusions with small values can be easily filtered by appropriately tuning the pruning parameter $\xi$.
\begin{figure}[!htb] 	
\centering
\includegraphics[width=0.95\textwidth]{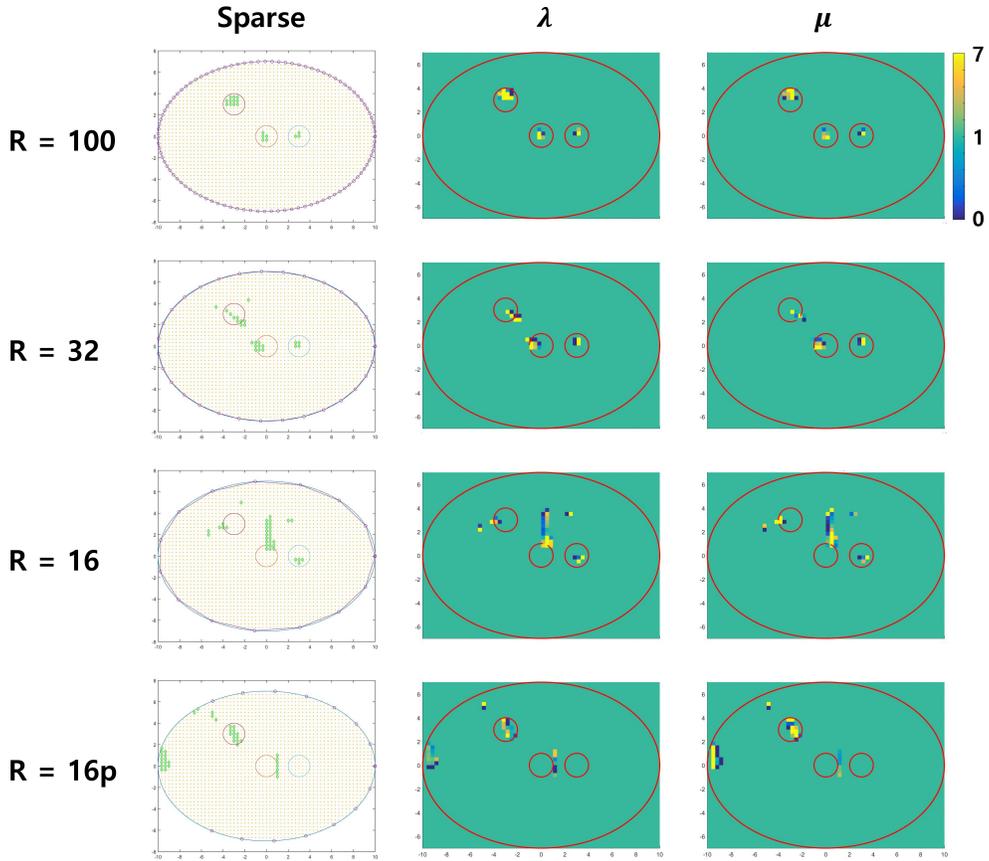}
\caption{Reconstruction of multiple disk-like inclusions. Left to Right: $\hat{D}$,  $\lambda$, and $\mu$. 
}
\label{fig:recon-3}
\end{figure}

\begin{figure}[!htb] 	
\centering
\includegraphics[width=0.95\textwidth]{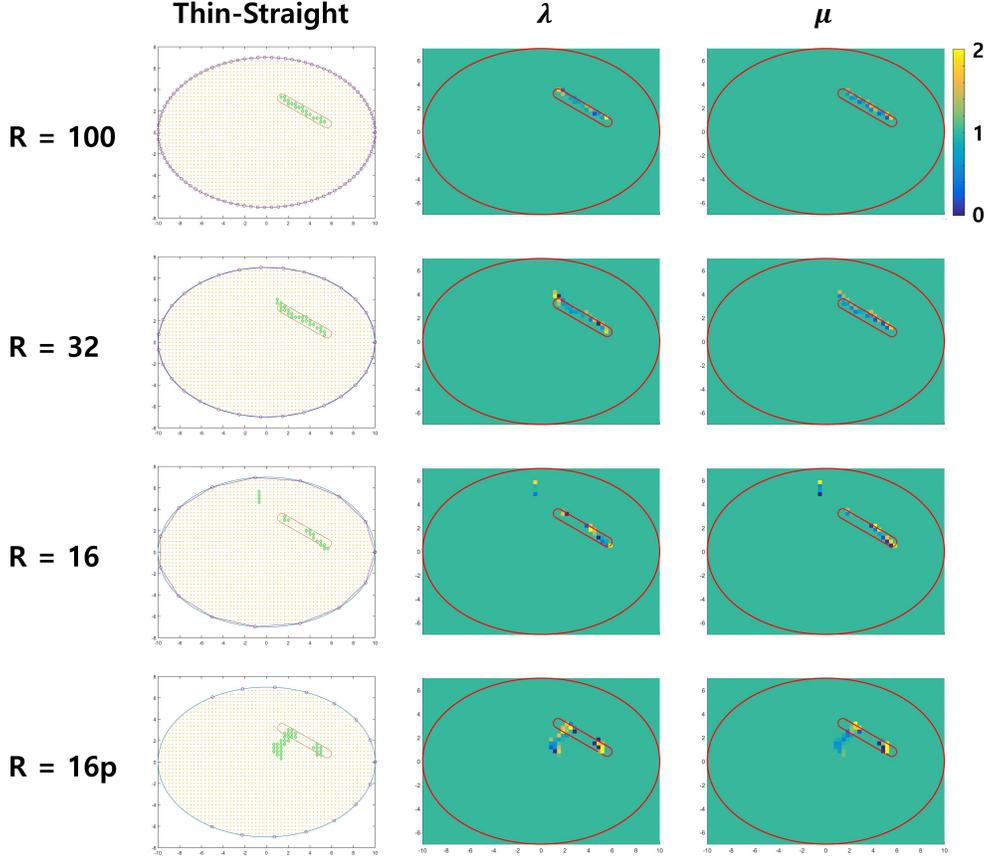}
\caption{Reconstruction of a thin-straight inclusion. Left to Right: $\hat{D}$,  $\lambda$, and $\mu$. 
} 
\label{fig:recon-thin1}
\end{figure}

\begin{figure}[!htb] 	
\centering
\includegraphics[width=0.95\textwidth]{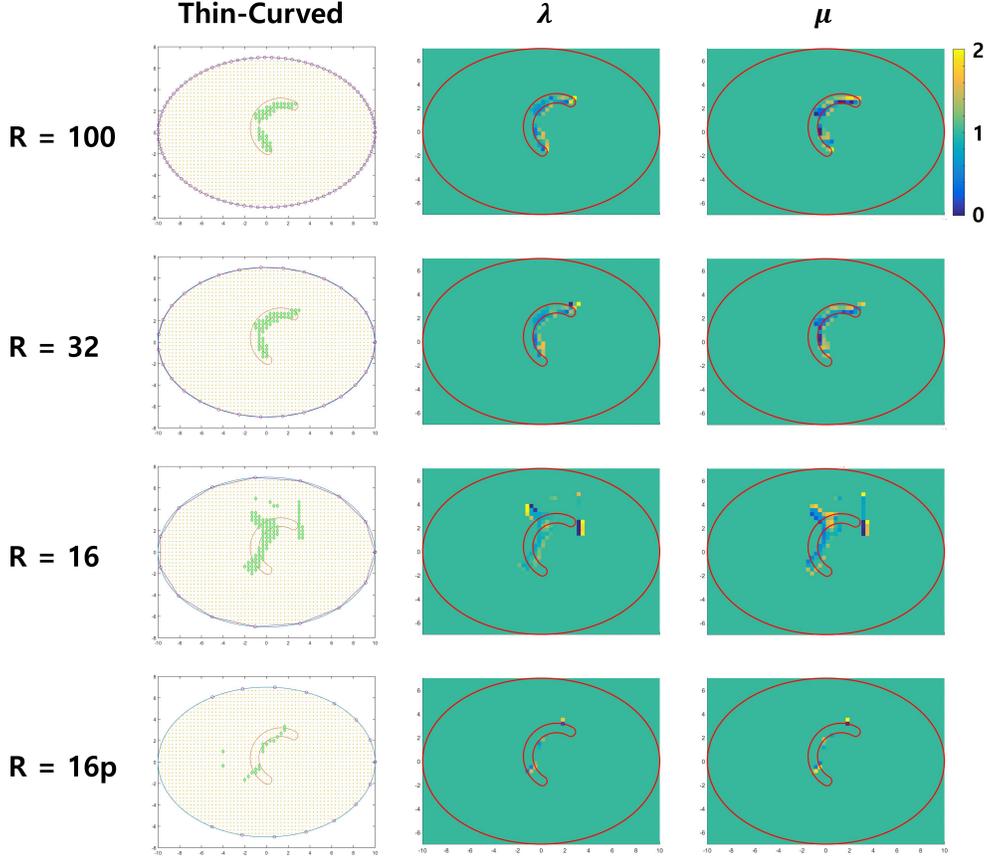}
\caption{Reconstruction of a thin-curved inclusion. Left to Right: $\hat{D}$,  $\lambda$, and $\mu$. 
}
\label{fig:recon-thin2}
\end{figure}

\begin{figure}[!htb] 	
\centering
\includegraphics[width=0.95\textwidth]{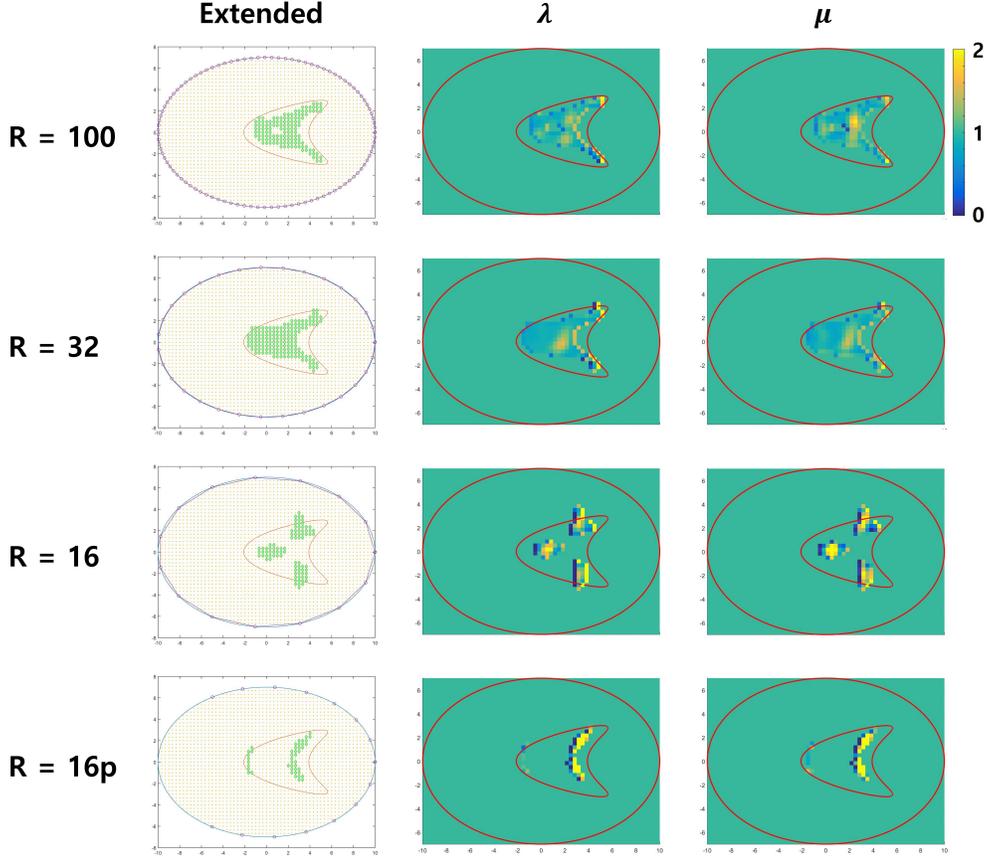} 
\caption{Reconstruction of an extended inclusion. Left to Right: $\hat{D}$,  $\lambda$, and $\mu$. 
}
\label{fig:recon-kite}
\end{figure}

\begin{figure}[!htb] 	
\centering
\includegraphics[width=0.95\textwidth]{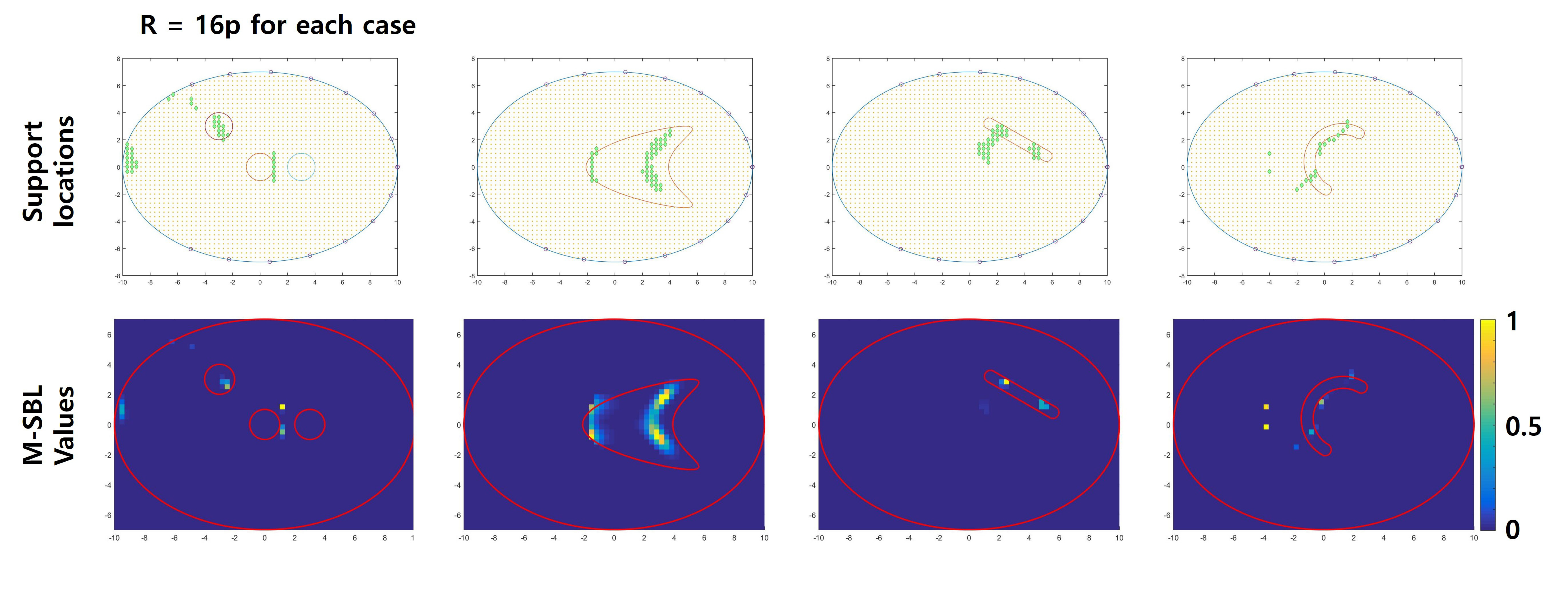} 
\caption{Support identification using M-SBL in deteriorated cases. Top:  Reconstructed support sets. Bottom: The normalized M-SBL values $\psi_\ell/\max_{i\in\{1,\cdots, L\}}(\psi_i)$.
}
\label{fig:MSBL-supp}
\end{figure}

\section{Conclusion}\label{s:conc}

A joint sparse recovery based direct algorithm was proposed to reconstruct the spatial support of multiple elastic inclusions and their material parameters using only a few measurements of the displacement over a very coarse grid of boundary points (in the sense of Nyquist sampling rate). The inverse problem for support detection was converted to a joint sparse recovery problem for internal data (linked to the displacement and strain fields inside the support set of the inclusions) by virtue of an integral formulation. The sparse signal recovery problem resulting therefrom was resolved  for an exact and unique solution by invoking a modified M-SBL algorithm with structural constraints. Then, using the leverage of the learned internal  information about the displacement field, a linear inverse problem for quantitative evaluation of material parameters was formulated. The resulting problem was  then converted to a noise robust constraint optimization problem, which was subsequently solved using the C-SALSA. The proposed imaging algorithm was computationally very efficient and was felicitous to demonstrate very accurate reconstruction  since it is non-iterative and does not require any linearization or computations of multiple forward solutions. 
The advantage is taken here of the learned internal data and the sparsity of the support set of the inclusions inside the elastic medium to reduce the mathematical ill-posedness of the underlying inverse problem. In fact, the recovery of such an information is very novel and pertinent. This additional information compensates for the under-determined data and therefore renders stability to the reconstruction framework. In addition, since no linearization or simplifying approximations are used,  the proposed technique provides reconstruction with better resolution and quality than classical techniques. However, a more sophisticated  quantitative mathematical analysis is certainly necessary in order to ascertain the stability and resolution properties of the proposed framework in terms of the relative size of the inclusion, the number of measurement fields, the number and the placement of the measurement points on the  boundary, and the aperture size.  This will be the subject of future investigations. Albeit, the elastostatic problem is  undertaken in this article, the quasi-static or time-harmonic elasticity problems are also amenable to the same treatment with minor changes. Moreover, the elasticity imaging problem in the so-called quasi-incompressible regime can also be dealt with and will be investigated in future.

\appendix
\section{Evaluation of integral kernals}\label{Append}

Let us provide the explicit expressions for different kernels involved in our integral formulation and those required to compute the sensing matrix of the reconstruction framework. For brevity, only the two dimensional case is entertained. 

Following identities will be handy in ensuing calculations. For all $i,j\in\{1,2\}$ and $\bx,\by\in\RR^2$, such that  $\bx\neq\by$,  
\begin{equation}
\begin{cases}
\ds\frac{\partial}{\partial y_i}\ln|\bx-\by| 
= -\frac{x_i-y_i}{|\bx-\by|^2},
\\
\ds\frac{\partial^2}{\partial y_i\partial y_j}\ln|\bx-\by|
= -2\frac{(x_i-y_i)(x_j-y_j)}{|\bx-\by|^4}+\delta_{ij}\frac{1}{|\bx-\by|^2},
\\
\ds \Delta_\by\ln|\bx-\by|
=0.
\end{cases}\label{identities}
\end{equation}

\subsection{Surface traction of Kelvin matrix and  boundary integral operators}

Recall from  \cite[Appendix A]{AKLL} that, for all $\bx,\by\in\RR^2$, $\bx\neq \by$ and $i,j\in\{1,2\}$, 
\begin{align}
\left(\frac{\partial\bGam}{\partial\bnu_\by}(\bx,\by)
\right)_{ij}
=&\left[a\delta_{ij}+b\frac{(x_i-y_i)(x_j-y_j)}{|\bx-\by|^2}\right]\left(\sum_{k=1}^2\frac{\nu_k(x_k-y_k)}{|\bx-\by|^2}\right)
\nonumber
\\
&-a\frac{\nu_j(x_i-y_i)-\nu_i(x_j-y_j)}{|\bx-\by|^2},
\label{GammaTraction}
\\
\nu_i=&[\bnu]_i,
\quad
a:=-\frac{\mu_0}{2\pi(\lambda_0+2\mu_0)},
\quad\text{and}\quad
b:=-\frac{\lambda_0 + \mu_0}{\pi(\lambda_0+2\mu_0)}.
\end{align}
 
Therefore, thanks to \eqref{GammaTraction},  operators $\mathcal{S}_\Omega$, $\mathcal{D}_\Omega$  and  $\mathcal{K}_\Omega$  can be evaluated as
\begin{align*}
&\mathcal{S}_\Omega[\bphi](\bx):= \ds\left(\sum_{i,j=1}^2 \int_{\partial \Omega}\left[\bGam(\bx-\by)\right]_{ij}\cdot\left[\bphi(\by)\right]_jd\sigma(\by)\right)\mathbf{e}_i,\quad\,\bx\in\RR^2\setminus\partial\Omega,
\\ 
&\mathcal{D}_\Omega[\bphi](\bx):= \ds\left(\sum_{i,j=1}^2 \int_{\partial \Omega}\left[\frac{\partial}{\partial\bnu_\by}\bGam(\bx-\by)\right]_{ij}\cdot\left[\bphi(\by)\right]_jd\sigma(\by)\right)\mathbf{e}_i,\quad\,\bx\in\RR^2,
\\ 
&\Kcal_\Omega[\bphi](\bx):= \ds\left(\sum_{i,j=1}^2{\rm p.v.}\int_{\partial \Omega}\left[\frac{\partial}{\partial\bnu_\by}\bGam(\bx-\by)\right]_{ij}\cdot\left[\bphi(\by)\right]_jd\sigma(\by)\right)\mathbf{e}_i,\quad{\rm a.e.}\quad\,\bx\in\partial\Omega,
\end{align*}
for all $\bphi\in L^2_\Psi(\partial\Omega)$.

\subsection{Divergence of Kelvin matrix}

It is reminded that (see, e.g., \cite[Lemma 3.3.2]{ammaribook})
\begin{align}
\nabla_\by\cdot\bGam(\bx,\by)=\frac{1}{(\lambda_0+2\mu_0)}\nabla_\by\Phi(\bx,\by)=\frac{1}{(\lambda_0+2\mu_0)}\sum_{i=1}^2\frac{\partial}{\partial y_i}\Phi(\bx,\by)\mathbf{e}_i,
\label{DivGamma1}
\end{align}
where $\Phi(\bx,\cdot):\RR^2\to\RR$, for fixed $\bx\in\RR^2$, is the fundamental solution to the Laplace equation in $\RR^2$, that is, 
$\Delta_\by\Phi(\bx,\by)= \delta_\bx(\by)$,  for all $\bx,\by\in\RR^2$
and is given by 
\begin{eqnarray*}
\Phi(\bx,\by)= \ds\frac{1}{2\pi}\ln |\bx-\by|,\quad \forall \bx,\by\in\RR^2, \quad \bx\neq \by.
\end{eqnarray*}
After fairly easy manipulations and using identities  \eqref{identities} in \eqref{DivGamma1}, one arrives at
\begin{align*}
\nabla_\by\cdot\bGam(\bx,\by)=\frac{a}{\mu_0}\sum_{i=1}^2\frac{(x_i-y_i)}{|\bx-\by|^2}\mathbf{e}_i,\quad \bx,\by\in\RR^2, \quad \bx\neq\by.
\end{align*}

\subsection{Strain of Kelvin matrix}

In order to calculate $\Big[\mathcal{E}[\bGam(\bx,\cdot)](\by)\Big]_{ijk}$, express $\bGam$ as
\begin{align*}
\gamma_{ij}(\bx-\by)= \alpha\delta_{ij}\ln|\bx-\by|+\beta(x_i-y_i)\frac{\partial}{\partial y_j}\ln|\bx-\by|,\quad \bx,\by\in\RR^2,\quad\bx\neq\by,
\end{align*}
where $\alpha$ and $\beta$ are given by \eqref{AlphaBeta}.
Therefore, 
\begin{align*}
&\frac{\partial}{\partial y_k}
\gamma_{ij}(\bx-\by)
\\
=&
\phantom{-1}\alpha\delta_{ij}\frac{\partial}{\partial y_k}\ln|\bx-\by|+\beta\frac{\partial}{\partial y_k}(x_i-y_i)\frac{\partial}{\partial y_j}\ln|\bx-\by|+\beta(x_i-y_i)\frac{\partial^2}{\partial y_k\partial y_j}\ln|\bx-\by| 
\\
=& -\alpha\delta_{ij}\frac{(x_k-y_k)}{|\bx-\by|^2}+\beta\delta_{ik}\frac{(x_j-y_j)}{|\bx-\by|^2}+
\beta\delta_{jk}\frac{(x_i-y_i)}{|\bx-\by|^2}-2\beta\frac{(x_i-y_i)(x_j-y_j)(x_k-y_k)}{|\bx-\by|^4}.
\end{align*}
Consequently, $\Big[\mathcal{E}[\bGam(\bx,\cdot)](\by)\Big]_{ijk}$ can be calculate, for all $i,j,k \in\{1,2\}$, as 
\begin{align*}
&2\Big[\mathcal{E}[\bGam(\bx,\cdot)](\by)\Big]_{ijk}
=
\left(\frac{\partial\gamma_{ij}}{\partial y_k}(\bx,\by)+\frac{\partial\gamma_{ik}}{\partial y_j}(\bx,\by)\right)
\\
=&-\alpha\delta_{ij}\frac{(x_k-y_k)}{|\bx-\by|^2}+\beta\delta_{ik}\frac{(x_j-y_j)}{|\bx-\by|^2}+
\beta\delta_{jk}\frac{(x_i-y_i)}{|\bx-\by|^2}-2\beta\frac{(x_i-y_i)(x_j-y_j)(x_k-y_k)}{|\bx-\by|^4}
\\
&-\alpha\delta_{ik}\frac{(x_j-y_j)}{|\bx-\by|^2}+\beta\delta_{ij}\frac{(x_k-y_k)}{|\bx-\by|^2}+
\beta\delta_{jk}\frac{(x_i-y_i)}{|\bx-\by|^2}-2\beta\frac{(x_i-y_i)(x_j-y_j)(x_k-y_k)}{|\bx-\by|^4} 
\\
=&
\left(\beta-\alpha\right)\left[\delta_{ij}\frac{(x_k-y_k)}{|\bx-\by|^2}+\delta_{ik}\frac{(x_j-y_j)}{|\bx-\by|^2}\right]+2\beta\delta_{jk}\frac{(x_i-y_i)}{|\bx-\by|^2}
-4\beta\frac{(x_i-y_i)(x_j-y_j)(x_k-y_k)}{|\bx-\by|^4}.
\end{align*}
Finally, by remarking that $a=\beta-\alpha$, one arrives at 
\begin{align*}
\Big[\mathcal{E}[\bGam(\bx,\cdot)](\by)\Big]_{ijk}
=&
\frac{a}{2}\left[\delta_{ij}\frac{(x_k-y_k)}{|\bx-\by|^2}+\delta_{ik}\frac{(x_j-y_j)}{|\bx-\by|^2}\right]+\beta\delta_{jk}\frac{(x_i-y_i)}{|\bx-\by|^2}
\\
&-2\beta\frac{(x_i-y_i)(x_j-y_j)(x_k-y_k)}{|\bx-\by|^4}.
\end{align*}

\end{document}